\documentclass{article}
\PassOptionsToPackage{square,numbers, compress}{natbib}

\usepackage[utf8]{inputenc} % allow utf-8 input
\usepackage[T1]{fontenc}    % use 8-bit T1 fonts
\usepackage{hyperref}       % hyperlinks
\usepackage{url}            % simple URL typesetting
\usepackage{booktabs}       % professional-quality tables
\usepackage{amsfonts}       % blackboard math symbols
\usepackage{nicefrac}       % compact symbols for 1/2, etc.
\usepackage{microtype}      % microtypography

\usepackage{amsmath}

\usepackage[table]{xcolor}
\usepackage{graphicx}
\usepackage{booktabs} % for professional tables

\usepackage{graphicx}
\usepackage{dblfloatfix}
\usepackage{url}
\usepackage{pdflscape}
\usepackage{hyperref}
\usepackage{multirow}
\usepackage{multicol}
\usepackage[english]{babel}
\usepackage{algorithm}
\usepackage{algorithmic}
\usepackage{comment}
\usepackage{latexsym, amsmath,amssymb, amsthm, xpatch} % mathtools
\usepackage{color}
\usepackage{float}
\usepackage{graphics, epsfig, epstopdf, tabularx, footnote, booktabs,pbox, tabularx}
\usepackage{amsfonts}

\usepackage{enumitem}

\usepackage{natbib}
\usepackage{tikz-cd} %commutative diagrams

\usepackage{amsmath,hyperref}
\usepackage{pdflscape}
\usepackage{verbatim}
\usepackage{amsfonts}

\usepackage{epstopdf}% To incorporate .eps illustrations using PDFLaTeX, etc.
 
\usepackage{amsthm} 
\theoremstyle{plain}% Theorem-like structures
\newtheorem{theorem}{Theorem}

\newtheorem{lemma}[theorem]{Lemma}
\newtheorem{proposition}[theorem]{Proposition}
\newtheorem{assumption}[theorem]{Assumption}

\theoremstyle{definition}
\newtheorem{definition}[theorem]{Definition}
\newtheorem{example}[theorem]{Example}

\theoremstyle{remark}

\usepackage[colorinlistoftodos,bordercolor=orange,backgroundcolor=orange!20,linecolor=orange,textsize=scriptsize]{todonotes}

\usepackage{dcolumn}

\newcolumntype{d}[1]{D{.}{\cdot}{#1} }

\newcommand{\inTR}[1]{}

\usepackage{algorithmic}

\DeclareGraphicsExtensions{.pdf,.jpeg,.png}
\usepackage{algorithmic}
\usepackage{array,color}
\usepackage{dblfloatfix}
\usepackage{url}

\usepackage{mathtools, stmaryrd}
\newcommand{\nn}[1]{\llbracket #1 \rrbracket} % a notation for intervals of integers
\DeclareMathOperator{\Ima}{Im} % a notation for the range of a linear operator or matrix
\DeclareMathOperator*{\argmax}{arg\,max}
\DeclareMathOperator*{\argmin}{arg\,min}
\newcommand{\tto}{\rightrightarrows}

\makeatletter
\newcounter{proofpart}
\xpretocmd{\proof}{\setcounter{proofpart}{0}}{}{}
\newcommand{\proofpart}[1]{%
  \par
  \addvspace{\medskipamount}%
  \stepcounter{proofpart}%
  \noindent\textit{Part \theproofpart: #1}\par\nobreak\smallskip
  \@afterheading
}
\makeatother % a cooked-up solution to divide a proof in parts

\newcommand{\vertiii}[1]{{\left\vert\kern-0.25ex\left\vert\kern-0.25ex\left\vert #1 
    \right\vert\kern-0.25ex\right\vert\kern-0.25ex\right\vert}} % a nice operator norm

\title{On-line Non-Convex Constrained Optimization}

\author{
  Olivier Massicot \\
  University of Illinois at Urbana-Champaign \\
  \texttt{om3@illinois.edu} \\
  \And
  Jakub Marecek \\
  IBM Research -- Ireland
  \texttt{jakub.marecek@ie.ibm.com} 
}

\author{Olivier Massicot, Jakub Marecek}

\begin{document}

\maketitle

\begin{abstract}
Time-varying non-convex continuous-valued non-linear constrained optimization is a fundamental problem. 
We study conditions wherein a momentum-like regularising term allow for the tracking of local optima by considering an ordinary differential equation (ODE). 
We then derive an efficient algorithm based on a predictor-corrector method, to track the ODE solution.
\end{abstract}

\section{Introduction}
	
Most problems in on-line learning are time-varying and many are non-convex. This is easy to see: consider the fact that each new observation %arriving
amends the instance of the model-estimation problem, e.g., its least-squares objective. Further, notice that while we often use convex relaxations, forecasting under assumptions related to linear dynamical systems,  forecasting under assumptions related to low-rank models, using auto-encoders, or indeed, much of any deep-learning techniques is inherently non-convex. For many problems, strong convexifications are not available.

Nevertheless, on-line non-convex constrained optimization has been largely neglected, so far. This includes both the study of structural properties of trajectories of feasible points of optimization problems, whose coefficients are time-varying, as well as algorithmic approaches to such problems. 
%Only recently, \cite{fattahi2019} have asked, when do on-line non-convex optimization problems have a map from continuous time $[0, T]$ to the solutions, such that the map is continuously differentiable, and at time $T$, the solution lies within the region of attraction of an optimum. This motivated our present study, starting with a more general question: when it is possible to have a map from continuous time $[0, T]$ to the solutions, at all?
Only recently, \cite{fattahi2019} have asked whether a momentum penalty could help track solutions of on-line non-convex optimization problems, and whether to expect the produced solution to converge to a global optimum after switching to the time-invariant problem. 
This motivated our present study, starting with a more general question: when it is possible to track current solutions of the unpenalized problem?

We provide some of the first structural results for on-line non-convex optimization in a rather general setting, wherein we allow for both equality and  inequality constraints, and isolated time discontinuities. % of the  right-hand sides of the constraints.
%In doing so, they consider a setting where there are only equality constraints and right-hand sides are continuously-varying. 
 In particular, we consider a continuous-time setting:% with a finite time horizon $T>0$: 
% \begin{align}
% \label{eq:inequalityproblem}
% 	\inf_{x\in \mathbb{R}^n} &  ~ f(x,t)\\
% 	\notag
% 	\text{ s.t. } & h_i(x) = d_i(t), \quad i = 1,\hdots,p \\
% 	\notag
% 	& g_j(x) \leq e_j(t), \quad j = 1,\hdots,q.
% \end{align}
\begin{align}
\label{eq:inequalityproblem}
	\inf_{x\in \mathbb{R}^n} &  ~ f(x,t)\\
	\notag
	\text{ s.t. } & h(x,t) = 0 \\
	\notag
	& g(x,t) \leq 0,
\end{align}
	where 
\begin{itemize}
% \item the objective function $f(x,t)$ varies over time $t\in [0,T]$ and $f: \mathbb{R}^n \times [0,T] \longrightarrow \mathbb{R}$
% are continuously differentiable 
% and locally Lipschitz
\item the objective function $f$ is twice continuously differentiable in $x \in \mathbb R^n$, and continuous in $t \in \mathbb R_+$ %(as well as its derivatives with respect to $x$) 
except possibly for some isolated times, 
% \item left-hand sides $g_j, h_i: \mathbb{R}^n \longrightarrow \mathbb{R}$, $d_i: \mathbb{R}_+ \longrightarrow \mathbb{R}$ are continuously differentiable and locally Lipschitz for for $i=1,\hdots,p$ and $j=1,\hdots,q$.
% $h(x)$ denotes $[h_1(x),\hdots,h_p(x)]^\top$
% and 
% $g(x)$ denotes $[g_1(x),\hdots,g_q(x)]^\top$.
% \item right-hand sides $d_i, e_j$ of the constraints vary over time $t\in [0,T]$ and there are only finitely many discontinuities.
% $d(t)$ denotes $[d_1(t),\hdots,d_p(t)]^\top$ and 
% $e(t)$ denotes $[e_1(t),\hdots,e_q(t)]^\top$. 
\item the constraints are respectively $p$ and $q$ dimensional (where $\leq 0$ is understood as `belongs to the nonpositive orthant'), both $h$ and $g$ are twice continuously differentiable in $x$ and continuously differentiable in $t$%(as well as their derivatives with respect to $x$)
, again except potentially for some isolated times, 
% \item the problem is feasible for all $t\in[0,T]$ and $f$ is coercive.% uniformly bounded from above and below for all feasible solutions.
\item at all time $t \in \mathbb R_+$, the problem is feasible and $f(\cdot,t)$ is coercive.% uniformly bounded from above and below for all feasible solutions.
\end{itemize}

Furthermore, we provide algorithms benefiting from our structural insights. 
In particular, we present 
 a predictor-corrector method
for integrating a related ordinary differential equation (ODE). 
As usual, the ``prediction'' step fits a function to values of the function and its derivatives, while the  ``corrector'' step refines the approximation. % using the Hessian.
Interestingly, however, we can avoid the need to perform \emph{any} matrix inversion or multiplication, %\cite[Chapter 10]{ramm2013dynamical},
which reduces the per-iteration complexity %for problems %with $n \times n$ Hessian of $s$ non-zeros
from $O(n^3)$ to $O(n^2)$. % Notice that $s$ is clearly at most $n^2$, while the inverse of a sparse Hessian can be fully dense.

%The paper is organised as follows: We formalise the problem in Section~\ref{sec:problem_statement}. In Section \ref{sec:method} ..In Section \ref{sec:anal}, ... Finally, in Section \ref{sec:exp}, ...

\section{Related Work}

The history of study of on-line (or, interchangeably, time-varying) optimization problems goes back at least to Bellman \cite{bellman1953bottleneck}. 
Clearly, one could consider a ``sampling'' approach and solve \eqref{eq:inequalityproblem} for finitely many times $t$. 
This may not, however, be representative of the actual optimal trajectory, depending on  %regularity conditions on %$f,g,h$, and 
the time step. For a small time-step, it can become very time-consuming. 

When $f,g$ are convex and $h$ is affine, there are many elegant results, which reduce the expense by showing that even a small number of iterations of various numerical methods may be sufficient in ``warm-starting'' the sampling approach, wherein the solution obtained for one off-line problem is considered as a candidate for the next off-line problem. Much of the early work is associated with the term continuous linear programming \cite{lehman1954continuous,tyndall1965duality,levinson1966class}.
% grinold1969continuous,buie1973numerical
%,luo1998new,fleischer2005efficient,bampou2012polynomial}, wherein at each time, the problem was convex. 
 Subsequently, on-line convex optimization \cite{bubeck2011introduction,shalev2012online,hazan2016introduction} has developed in machine learning, with a particular focus on on-line gradient descent, and time-varying convex optimization \cite[e.g.,]{7469393,7993088,8631190} has been studied in signal processing. 

When $f,g$ are not convex or $h$ is not affine, very little is known in general. 
In on-line learning, much of the recent work  \cite[to continue our example of learning under assumptions of a linear dynamical system]{arima_aaai,kozdoba2019line,sarkar2019} have focused on convexifications of the non-convex problem. 
In control theory, and especially  
within %community \cite[e.g.]{YAKUBOVICH1997,mordukhovich1999existence} 
extremum seeking \cite{ariyur2003real}, there has been focus on the case of an \emph{unknown} non-convex $f, g$ and affine $h$.
%We refer to the classic textbook of \cite{ariyur2003real} on extremum seeking.
In signal processing and computer vision, the use of low-rank assumptions
in the processing of video data has led to the study of problems exhibiting restricted strong convexity.
We refer to the August 2018 special issue of the Proceedings of the IEEE \cite{8425660} and the related issue of Signal Processing Magazine \cite{8398586} for up-to-date surveys of the work within signal processing.
We stress that this work deals with very special, structured problems.

In this paper, we suggest that tracking a solution to the non-convex problem may offer an appealing alternative. We show that this is scalable to high-dimensional settings, thanks to a low per-update run-time.  
In particular, we consider the small limit of time step in a gradient method or Newton method acting on the time-varying optimization problem, captured in an ordinary differential equation (ODE). 
This builds upon a long and rich history of work on  ODE-based models of time-invariant (off-line) optimization.
There, starting from the early analyses of methods for solving non-linear equations in mathematical analysis and operator theory % ZHIDKOV1967,branin1972widely,schropp1997note,schropp1995using,Airapetyan1999
\cite[e.g.]{gavurin1958nonlinear,polyak1964some,boggs1971solution}, convergence results relied on the stability of related ODEs.  More recently, this approach is sometimes associated with the term  dynamical-systems method (DSM), cf. \cite{ramm2013dynamical}. % schropp2000dynamical
In mathematical optimization, the recent uses  
%wibisono2016variational
 \cite[e.g.]{lessard2016analysis,su2014differential,hardt2016train,scieur2016regularized,kim2016optimized,gurbuzbalaban2017convergence,taylor2017smooth,fazlyab2018analysis} consider the choice of  hyper-parameters of  first-order methods  so as to optimise the rate of convergence.
 % \cite{Drori2014}
These uses are, however, all related to time-invariant problems. 

The only use of the continuous-time limit of an iterative numerical method for the study of time-varying non-convex optimization, which we are aware of, is by  
 \cite{fattahi2019}, who define the notion of a spurious trajectory, which does not reach global optimum at the end of a finite time horizon, and focus on its absence or existence in the case where  
(1) there are only equality constraints, rather than inequalities, 
 (2) functions are continuously time-varying.
 %functions on the left-hand side are continuously twice differentiable and  (3) the right-hand sides are continuously-varying. 
 We relax assumptions (1-2), focus on tracking local optima, and provide efficient  
 discrete-time algorithms, 
 in addition to structural results
 as explained in the next section.
\section{Our Approach}

%The original optimization program \cite{fattahi2019} would only account for equality constraints, being posed as,
%\begin{align}
%	\inf_{x\in \mathbb{R}^n} &  ~ f(x,t)\\ \text{ s.t. } & h_i(x) = d_i(t), \quad i = 1,\hdots,p
%	\label{eq:originalproblem}
%\end{align}
	
We consider a ``tracking'' approach to time-varying non-convex optimization. 
In the case of a gradient method applied to a problem restricted to equalities (cf. \eqref{eq:generalpb} below) we obtain the following ODE: %with slack variables
\begin{align}
	%\boxed{
		\dot{x} = -\frac{1}{\alpha} \eta(x,t) - \theta(x,t)\dot{h}(x,t),%}
	%\tag{ODE$\geq$}
	\label{eq:odeineq}
\end{align}
where,
\begin{align}
	\theta(x,t) ~ \triangleq & ~~ \mathcal{J}_h(x,t)^\top(\mathcal{J}_h(x,t)\mathcal{J}_h(x,t)^\top)^{-1}, \\
	\eta(x,t) \triangleq & ~~ \left[I-\theta(x,t)\mathcal{J}_h(x,t)\right]\nabla_x f(x,t),
\end{align}
$\mathcal{J}_h$ is the Jacobian of $h$ with respect to the first variable and $\alpha$ a parameter embodying the momentum penalty. %of the left hand side of the constraints.	

%\cite{fattahi2019} have shown a solution of \eqref{eq:odeineq} exists and is unique, and that the discrete-time local solutions of \eqref{eq:generalpb} endowed with momentum are the solution of \eqref{eq:odeineq}, in the limit of short time steps. 

% Instead, we present a tracking approach. This approach relies on 
%  an ordinary differential equation (ODE) modelling a gradient method applied to the formulation \eqref{eq:equalityproblem} with slack variables:
% 	\begin{align}
% 	%\boxed{
% 		\dot{X} = -\frac{1}{\alpha} \eta(X,t) + \theta(X)\dot{d}%}
% 	\tag{ODE$\geq$}\label{eq:odeineq}
% 	\end{align}
% 	where $X = [x~z]^\top$, $H = [h~g]^\top$,
% 	\begin{align}
% 	\eta(X,t) := & ~~ \left[I-\mathcal{J}_H(X)^\top(\mathcal{J}_H(X)\mathcal{J}_H(X)^\top)^{-1}\mathcal{J}_H(X)\right]\nabla_X f(x,t), \\
% 	\theta(X) ~ := & ~~ \mathcal{J}_H(X)^\top(\mathcal{J}_H(X)\mathcal{J}_H(X)^\top)^{-1}.
% 	\end{align}
% 	and $\mathcal{J}_H(X)$ is the Jacobian of the left hand side of the constraints.
%Notice, in particular,  the similarities with (27) of \cite{gurbuzbalaban2017convergence} and (ODE) of 	\cite{fattahi2019}, who focussed on equality constraints.

% This can be simplified to:

% ***

% in the case of equalities.   /// It is much easier to use equalities all way through, and let inequalities be accounted for with slack variables

% Fattahi and coauthors \cite{fattahi2019} have shown that,
% \begin{enumerate}
%     \item \eqref{eq:odeineq} has a unique solution,
%     \item the solution 
% \end{enumerate}

At that point, one can take several avenues. \cite{fattahi2019}  suggested that one should like to focus on trajectories, which at the end $T$ of a finite time horizon $[0, T]$
are within the region of attraction of the global solution. 
This notion, which \cite{fattahi2019} call the non-spurious trajectory, can, 
however, lead to a time-integral
of the objective function over $[0, T]$, which is arbitrarily 
worse than a  time-integral of the objective function of the best possible trajectory. Also, \cite{fattahi2019} rely on $f,h$ changing fast enough hoping to escape spurious trajectories, but there may simply be no $\alpha$ corresponding to the rate of change of $f,h$ that would allow for the tracking of the non-spurious trajectory. 

Within the spirit of optimal control 
%and on-line optimization \cite{bubeck2011introduction,shalev2012online,hazan2016introduction},
one could seek a trajectory that minimizes the time-integral
of the objective function over $[0, T]$.
Done naively, this could take the form of trying to seek the global optimum at all times. 
However, in the worst case, 
where global optimality would alternate between two local minima separated by a ``fence'' of a global maximum, the naive algorithm trying
to minimize the time-integral could end up at the ``fence'' much of the time. Further, this example shows that the notion of ``spurious trajectory'' is somewhat arbitrary, depending on the time horizon.

Instead, we advocate to focus on tracking a local minimum. 
For many problems (in the class of NP over reals) there exists a test of global optimality of a point.
If the local solution is no longer a global solution, we may 
run a ``refinement'' step to ``recenter'' our trajectory and follow it from the output of the other method, albeit sometimes at a substantial computational cost, e.g., when $f, g, h$ are polynomial.
Hence, we focus on the ODE \eqref{eq:odeineq}, as it represents what one could achieve solving the off-line problem frequently enough. % \eqref{eq:generalpb} with momentum. 
In the next section, we will see our main contribution: how the solution of \eqref{eq:odeineq} tracks local solutions of the problem \eqref{eq:inequalityproblem} for $\alpha$ small enough.
Subsequently, we develop iterative algorithms
based on a discretization of time in solving of \eqref{eq:odeineq}.
These have several benefits, but most importantly, the solving of the ODE is easier than performing the gradient updates or running the Newton method, as one can avoid performing any matrix inversion.

% In the next section, we prove existence and uniqueness of a solution to the ODE, assuming $X_0$ is a solution of the program \eqref{eq:equalityproblem} at time $t=0$. Further, we prove that any sequence of discrete local trajectories $(X^{\Delta t})$ for which $\Delta t \to 0$, converges to the solution $X$:
% \begin{equation}
% \lim_{\Delta t \to 0} \max_{0 \leq k \leq \lceil \nicefrac{T}{\Delta t} \rceil} \| X_k^{\Delta t} - X(t_k) \| = 0.
% \end{equation}
%These structural results 
%justify the use of discretization of time and the use of an iterative method in this setting.

%\input{problem_statement}
%\input{analysis}
\section{Structural Results}

\subsection{Assumptions}

For simplicity, we might denote by $h_t$ the function $x \in \mathbb R^n \mapsto h(x,t)$ and likewise for $f$ and others. From now on, we will make the following assumption,
\begin{assumption}[Smoothness]
\label{as:functionclass}
%$f_t$ is coercive, twice continuously differentiable in $x$, continuous in $t$, $h_t$ is twice continuously differentiable in $x,t$, $\mathcal{J}_h$ is always full row-rank and $h_t^{-1}(\{0\})$ is never empty.
%Let $I = [0,T]$ with some $T > 0$. Let 
$f : \mathbb{R}^n \times \mathbb R_+ \to \mathbb{R}$ and $h : \mathbb{R}^n \times \mathbb R_+ \to \mathbb{R}^m$ are twice continuously differentiable in their first variable $x$, $f$ being continuous in its second variable $t$ and $h$ continuously differentiable in $t$, such that $f_t$ is coercive, at all time $t \geq 0$.
\end{assumption}
This assumption corresponds to a smooth and regular problem at all time, which also varies continuously over time. 
Initially, we will  focus on the problem with equalities and without discontinuities:
\begin{align}
\label{eq:generalpb}
	\inf_{x\in \mathbb{R}^n} &  ~ f(x,t)\\ \text{ s.t. } & ~h(x,t)=0%& h_i(x,t) = 0, \quad i = 1,\hdots,p
	\notag
\end{align}
over $t \geq 0$ with $h$ vector-valued 
for the clarity of the initial presentation of the structural results.
We add two more assumptions which guarantee that the problem is feasible, and the constraints are regular. We note that infeasibility is often a failure of modelling, e.g., a recourse decision not modelled.

\begin{assumption}[Feasibility]
\label{ass:feaseq}
Defining the constraint manifold, for all $t \geq 0$,
\[
    \mathcal M(t) \triangleq h_t^{-1}(\{0\}) = \{ x \in \mathbb{R}^n, ~h(x,t) = 0\},
\]
we assume that its never empty, namely \eqref{eq:generalpb} is feasible.
\end{assumption}

\begin{assumption}[LICQ]
\label{ass:cont}
We assume that for all $t \geq 0$ and $x \in \mathcal M(t)$, the constraints are linearly independent, namely $\mathcal{J}_{h_t}(x)$ is full row-rank, the differential $\mathrm{d}h_t(x)$ is surjective.
\end{assumption}

Once the key structural results will be stated, we will see that they are actually rather easy to extend to the setting with inequalities and isolated discontinuities.
In the case of inequalities, we would like to consider the problem,
\begin{align}
\label{eq:generalpbineq}
	\inf_{x\in \mathbb{R}^n} &  ~ f(x,t)\\ \text{ s.t. } &
	~h(x,t)=0 \notag \\
	& ~g(x,t)\leq0%& h_i(x,t) = 0, \quad i = 1,\hdots,p
	\notag
\end{align}
over $t \geq 0$ with $g,h$ vector-valued, and where $g(x,t) \leq 0$ denotes that all components of $g(x,t)$ are nonpositive. This problem can be simply transformed into \eqref{eq:generalpb}, although we need to generalize our assumptions.
\begin{assumption}[Smoothness for inequalities]
\label{as:functionclassineq}
In addition to Assumption \ref{as:functionclass}, assume that $g$ follows the same assumptions as $h$.
\end{assumption}

\begin{assumption}[Feasibility for inequalities]
\label{ass:feasineq}
For all $t \geq 0$, we assume that there exists $x \in \mathbb{R}^n$ such that $h(x,t) = 0, ~g(x,t) \leq 0$, namely \eqref{eq:generalpbineq} is feasible.
\end{assumption}

\begin{assumption}[LICQ for inequalities]
\label{ass:contineq}
We assume that for all $t \geq 0$ and $x \in \mathbb{R}^n$ such that $h(x,t) = 0, ~g(x,t) \leq 0$, the constraints are linearly independent (LICQ), namely the gradients of the active inequality constraints and the gradients of the equality constraints are linearly independent.
\end{assumption}

\subsection{First Definitions}

%We begin by stating our original motivation. One way to go about solving \eqref{eq:generalpb} is to solve the problem at many given times, hopefully often enough. Fattahi and coauthors \cite{fattahi2019}, introducing a regularising momentum, stumble upon an ODE whose solution represents the limit of that approach.

Considering on-line non-convex optimization is little studied, we seem to lack a shared language for describing our problem and approach to it. Let us hence introduce some definitions first:

\begin{definition}[$\delta$-partition]
% An increasing finite sequence of times $\tau = (\tau_k)_{k \in \nn{0,N}}$ of $[0,T]$ is a called a \emph{$\delta$-partition}, for $\delta >0$, if $\tau_0 = 0$, $\tau_N = T$ and for all $k \in \nn{1,N}$,
An increasing sequence of times $\tau = (\tau_k)_{k \in \mathbb N}$ is a called a \emph{$\delta$-partition}, for $\delta >0$, if $\tau_0 = 0$ and for all $k \in \mathbb N^*$,
\[
    \tau_k - \tau_{k-1} \leq \delta,
\]
and if,
\[
    \tau_k \to_k \infty.
\]
%When there will be no ambiguity, $N$ will denote the final index of $\tau$. 
The set of $\delta$-partitions is denoted $T_\delta$. 
The tightness of $\tau$ is the minimal $\delta$ such that $\tau \in T_\delta$, we will say $\tau$ is $\delta$-tight if $\tau \in T_\delta$.
\end{definition}
Then, for $\delta >0$ and $\tau \in T_\delta$ 
we could define a solution as any sequence 
%$(x_k)_{k \in \nn{0,N}} \in \mathbb{R}^n$
$(x_k)_{k \in \mathbb N}$ in $\mathbb R^n$ such that each $x_k$ solves the program \eqref{eq:generalpb} at time $t=\tau_k$. However, we might find a solution $x$ that varies rapidly this way, for example on a rather flat landscape, where local minima alternately become globally optimal. Physically, it might not be feasible to have a solution changing too abruptly, and computationally, this approach would be costly and would not take advantage of prior knowledge. To make the problem more amenable, we penalize change by the inclusion of a momentum term in the objective:
\begin{definition}[$(\alpha,)\tau$-solution]
Let $\delta, \alpha>0$ and $\tau \in T_\delta$ be a $\delta$-partition. We call \emph{$\alpha,\tau$-solution}, any sequence $x$ of $\mathbb{R}^n$ such that $x_0$ solves the program at $t=0$,
\begin{align}
\label{eq:problem0}
	\inf_{x\in \mathbb{R}^n} &  ~ f(x,0)\\
	\notag 
	\text{ s.t. } & ~ h(x,0) = 0, %h_i(x,0) = 0, \quad i = 1,\hdots,p
\end{align}
and each $x_k$ solves,
%, with $k \in \nn{1,N}$, solves,
\begin{align}
\label{eq:problemk}
	\inf_{x\in \mathbb{R}^n} &  ~ f(x,\tau_k) + \frac{\alpha}{2} \frac{\|x-x_{k-1}\|^2}{\tau_k - \tau_{k-1}}\\ 
	\notag
	\text{ s.t. } & ~h(x,\tau_k) = 0.% h_i(x, \tau_k) = 0, \quad i = 1,\hdots,p
\end{align}
$\alpha$ will be omitted since fixed, embodying here a reluctance to change, an inertia.
\end{definition}

The goal is now to find a $\tau$-solution for $\tau$ tight enough within computational constraints. As it is, this problem is no easier to solve than the original one. Nonetheless, one can notice that given enough regularity of the constraints (assumptions \ref{ass:feaseq} and \ref{ass:cont}), all solutions of program \eqref{eq:problemk} must satisfy KKT conditions. Namely, if $x$ is a $\tau$-solution, there must exist a sequence 
$(\lambda_k)_{k \in \mathbb N^*}$ of
%$(\lambda_k)_{k \in \nn{1,N}}$ of
$\mathbb{R}^p$, such that for all 
$k \in \mathbb N^*$,
%$k \in \nn{1,N}$,
\begin{align}
    \nabla_x f(x_k, \tau_k) + \alpha \frac{x_k - x_{k-1}}{\tau_k - \tau_{k-1}} + \mathcal{J}_h(x_k, \tau_k)^\top \lambda_k &= 0 \\
    h(x_k, \tau_k) &= 0.
    \label{eq:penproblemk}
\end{align}
Note that we can rewrite the left-hand side of constraint \eqref{eq:penproblemk} as,
\begin{equation}
\frac{h(x_k, \tau_k) - h(x_k, \tau_{k-1}) + h(x_k, \tau_{k-1}) - h(x_{k-1}, \tau_{k-1})}{\tau_k - \tau_{k-1}}.
\end{equation}
%equation \eqref{eq:penproblemk} starts to resemble,
% \begin{align}
%     \nabla_x f(x_{k-1}, \tau_k) + \alpha \frac{x_k - x_{k-1}}{\tau_k - \tau_{k-1}} + \mathcal{J}_h(x_{k-1})^\top \lambda_k &= 0 \\
%     \frac{h(x_k, \tau_k) - h(x_k, \tau_{k-1}) + h(x_k, \tau_{k-1}) - h(x_{k-1}, \tau_{k-1})}{\tau_k - \tau_{k-1}} &= 0,
% \end{align}
Informally, in the limit of tight $\tau$ and if $x, \lambda$ were continuous functions of time,
\begin{align}
    \nabla_x f(x(t), t) + \alpha \dot{x}(t) + \mathcal{J}_h(x(t),t)^\top \lambda(t) &= 0 \label{eq:odelambda} \\
    \mathcal{J}_h(x(t),t) \dot{x}(t) + h'(x(t),t) &= 0, \label{eq:odeconstraint}
\end{align}
where $h'$ denotes the partial derivative of $h$ with respect to time. For the sake of readability, we introduce the symbolic notation $J = \mathcal{J}_h(x(t),t) = \mathcal J_{h_t}(x)$. Following assumption \ref{ass:cont},
%If we assume that at all time the constraints are regular (LICQ), then 
$J$ is of full row-rank. Thus, $J J^\top$ is invertible and we can deduce from \eqref{eq:odelambda} and \eqref{eq:odeconstraint} that
\begin{equation}
    (J J^\top)^{-1} J \nabla_x f(x(t), t) - \alpha (J J^\top)^{-1} h'(x(t),t) + \lambda(t) = 0.
\end{equation}
Subsequently, we can substitute $\lambda$ in equation \eqref{eq:odelambda},
\begin{align}
    \dot{x}(t) = & - \frac{1}{\alpha}(I_n - J^\top (J J^\top)^{-1} J) \nabla_x f(x(t), t)
    \tag{ODE}\label{eq:ode} \\
    & - J^\top (J J^\top)^{-1} h'(x(t),t). \notag
\end{align}

This way, we obtain an ODE, where only the initial condition requires us to solve an optimization program. The rest of the tracking is a simple integration. One can rephrase Theorem 1 of \cite{fattahi2019} to say that this ODE solves the penalized program \eqref{eq:penproblemk} in the limit of tight $\tau$ over a finite horizon, under mild regularity conditions which are met as soon as assumptions \ref{as:functionclass}, \ref{ass:feaseq} and \ref{ass:cont} are satisfied.

% , we merely brushed a motivation so far. The major points are,
% \begin{enumerate}
%     \item proving there exists a unique solution to \ref{eq:ode}, assuming some regularity on $f,h$,
%     \item showing that all $x_\tau$ converge uniformly to the solution $x$ of \eqref{eq:ode}.
% \end{enumerate}
% These results are intertwined in their proofs even though their statements seem independent. For instance, an assumption on the regularity of all $\tau$-solutions would translate to an interesting property for the ODE.

In the following, we define,
\[
    P = I_n - J^\top (J J^\top)^{-1} J,
\]
it is the orthogonal projection on $\ker J$ as shown in the appendix (lemma \ref{lem:proj}).
% the next lemma proves. The proof is in the Appendix. %\ref{app:ode}.

% \begin{lemma}
% \label{lem:proj}
% $P = I_n - J^\top (J J^\top)^{-1} J$ is the orthogonal projection on $\ker J$, thus,
% \begin{equation}
%     I_n \succeq P \succeq 0.
% \end{equation}
% \end{lemma}
\subsection{Time Scales and Continuity of the Flow}

Beyond the fact that the discrete problem turns out to match the continuous \eqref{eq:ode} in the small limit of $\tau$, we can see that in the small limit of $\alpha$, the solution of \eqref{eq:ode} converges to a local solution of the original problem \eqref{eq:generalpb}.

%One could argue that 
If \eqref{eq:ode} does not seem to distinguish between minima, maxima, nor even saddle points, as it provides only an approximation of sequences satisfying \emph{necessary} first-order conditions of optimality, %However,% we can write,
% \begin{align}
%     \frac{\mathrm{d}}{\mathrm{d}t} & f(x(t),t)
%     =~ \nabla_x f(x(t),t)^\top \dot{x}(t) + \frac{\partial}{\partial t} f(x(t),t) \\
%     &=~ - \frac{1}{\alpha} \nabla_x f(x(t),t)^\top (I_n - J^\top (J J^\top)^{-1} J) \nabla_x f(x(t), t) \label{eq:diffmaxmin} \\
%     \notag 
%     &\quad - \nabla_x f(x(t),t)^\top  J^\top (J J^\top)^{-1} h'(t)
%     + \frac{\partial}{\partial t} f(x(t),t)
% \end{align}
\begin{align}
    \frac{\mathrm{d}}{\mathrm{d}t} f(x(t),t)
    =&~ \nabla_x f(x(t),t)^\top \dot{x}(t) + \frac{\partial}{\partial t} f(x(t),t) \notag \\
    =&~ - \frac{1}{\alpha} \nabla_x f(x(t),t)^\top P \nabla_x f(x(t), t) \label{eq:diffmaxmin} \\
    \notag 
    &~ - \nabla_x f(x(t),t)^\top  J^\top (J J^\top)^{-1} h'(t) \\
    &~ + \frac{\partial}{\partial t} f(x(t),t). \notag
\end{align}
The term on line \eqref{eq:diffmaxmin} is non-positive since $P \succeq 0$ (\emph{cf.} Lemma \ref{lem:proj}), while the others remain `out of control' as induced by time variations of the demand and objective. Further, if we switch to the time-invariant system at time $T$, we get for $t \geq T$,
\begin{equation}
    \frac{\mathrm{d}}{\mathrm{d}t} f(x(t),T)
    = - \frac{1}{\alpha} \nabla_x f(x(t),T)^\top P \nabla_x f(x(t), T) \leq 0,
\end{equation}
with equality if and only if $\nabla_x f(x(t), T) \in (\ker J)^\perp = \Ima J^\top$, namely if and only if $x(t)$ satisfies the necessary KKT conditions (provided it is a regular point). Thus, informally $f$ is a Lyapunov function which brings the switched system to local minima.

%\subsection{Continuity of the flow}

Furthermore, by a simple time-rescaling $s = \nicefrac{t-t_0}{\alpha}$, \eqref{eq:ode} amounts to,
\begin{equation}
     \dot{x}(s) = - P \nabla_x f(x(s), \alpha s+t_0) - \alpha J^\top (J J^\top)^{-1} h'(\alpha s+t_0),
     \tag{ODE$_\alpha$}\label{eq:odescaled}
\end{equation}
$P, J$ taken at $(x(s), \alpha s + t_0)$. If now we take $\alpha$ to $0$, by continuity of the flow on $\alpha$, we expect solutions of \eqref{eq:odescaled} to resemble the solutions of the following autonomous ODE,
\begin{equation}
     \dot{x}(s) = - P \nabla_x f(x(s), t_0),
     \tag{ODE$_0$}\label{eq:odescaledflat}
\end{equation}
where $P$ is now taken at $(x(s), t_0)$. %By convenience, for $t_0 \in [0,T]$, we may use the notation
% \[
%     f_{t_0} = f(\cdot, t_0), \textrm{ \emph{etc.}}
% \]

By a Lyapunov argument, we will show that 
%(1) the critical points of \eqref{eq:odescaledflat} are exactly the points satisfying KKT necessary conditions, (2) 
a solution of \eqref{eq:odescaledflat} with almost any initial condition converges to a local minimum. % (depending on the initial condition). 
Then, if we take $\alpha$ small enough, by continuity of the flow with respect to the parameter $\alpha$, we expect the solution to $\eqref{eq:odescaled}$ to converge to a time-variant local (possibly spurious) minimum of \eqref{eq:generalpb}. %This statement is made rigorous in the appendix. 
In the next subsection, we will open the way to study this ODE at $\alpha = 0$.%We give a differentiable-manifold version of the continuous dependence on initial condition, since our solutions are evolving on the manifold $h_{t_0} = 0$.

\subsection{The Lyapunov Argument}

\begin{definition}
Let $f : %U \subset 
\mathbb{R}^n \to \mathbb{R}$, $g : %U \subset
\mathbb{R}^n \to \mathbb{R}^n$ and $h : %U \subset 
\mathbb{R}^n \to \mathbb{R}^m$ be continuously differentiable. % on $U$ open. 
We say $f$ is a Lyapunov function of,
\begin{equation}
    \label{eq:lyapode}
    \dot{x} = g(x),
\end{equation}
if $(\nabla f^\top g) (\mathbb R^n) \subset \mathbb{R}_-$, and $h$ is a conserved quantity if $(\nabla h^\top g)(\mathbb R^n) = \{0\}$. Further, we define the stationary set $\mathcal{C} = (\nabla f^\top g,h)^{-1}(\{(0,0)\})$, namely critical points on the constraint manifold $\mathcal{M} = h^{-1}(\{0\})$, it is a closed set as continuous preimage of a closed set. We complete the definition of a Lyapunov function by requiring that,
\[
(\nabla f^\top g)^{-1} (\{0\}) \subset \mathcal C.
\]%, and $\mathcal{M}_{++} \subset \mathcal{L}$ the subset of strict local maxima of $f$, isolated in $\mathcal{L}$.
\end{definition}

% \begin{definition}
% A set $S$ is said to be path-connected if for any two points $x,y \in S$, there exists a continuous function $\phi : [0,1] \to S$ (a path) such that $\phi(0) = x$ and $\phi(1) = y$. The path-connected subsets of a set $S$ are ordered by inclusion, the maximal elements are called the path-connected components of $S$ and are nonempty, disjoint, closed in $S$ and path-connected.
% \end{definition}

% \begin{definition}
% Let $S \subset U$ be a set, we define,
% \[
%     S^\epsilon = \{ x \in U, ~d(x,S) < \epsilon \}.
% \]
% \end{definition}

% \begin{definition}
% A set $S$ is said to be connected if it cannot be written as two nonempty disjoint subsets, open in $S$. Informally, $S$ is connected if it is `in one piece'. The connected subsets of a set $S$ are ordered by inclusion, the maximal elements are called the connected components of $S$ and are nonempty, disjoint, closed in $S$ and connected.
% \end{definition}

\begin{definition}[Jail]
A level set of $f$ under constraint $h=0$ is a set of the form,
\[
    L = (f,h)^{-1}((-\infty,c]\times\{0\}) = f^{-1}((-\infty, c]) \cap \mathcal{M},
\]
with $c \in \mathbb{R}$ called the edge of $L$, namely the set of $x \in \mathbb R^n$ such that $f(x) \leq c$ while $h(x) = 0$. A \emph{jail} 
%n isolated level set 
is a compact subset $K \subset L$ such that $L \setminus K$ is closed, this isolates $K$ from the rest of $L$.
% there exists an open set $V$ such that,
% \[
%     L \cap V = K.% \subset V.
% \]
If furthermore,
\[
    \mathcal{C} \cap K = \argmin_K f,
\]
the jail is said to be contracting.

%Abusing the notation slightly, we will say that $S$ is a path-connected lower level set (of $f$ under constraint $h$), if it is a path-connected component of some lower level set.
\end{definition}

\begin{lemma}
\label{lem:lowerlevel}
If $f : \mathbb{R}^n \to \mathbb{R}$ is continuous and coercive, its lower level sets, under constraint $h=0$, are compact.% in $U$.
\end{lemma}
With this lemma, it suffices that $K$ is closed and isolated from $L$ to be a jail. 
%an isolated lower level set. 
Next, we prove in appendix the following:

\begin{proposition}
\label{prop:lyap}
%If $f$ is a coercive Lyapunov function of \eqref{eq:lyapode} and $K \subsetneq U$ is a path-connected lower level set of $f$ under constraint $h$, then
% If $f$ is a coercive Lyapunov function of \eqref{eq:lyapode} and $K$ is an isolated level set of $f$ under constraint $h$, then
If $f$ is a Lyapunov function of \eqref{eq:lyapode} and $K$ is
a jail, 
%an isolated level set of $f$ under constraint $h$,
then,
\begin{enumerate}
    \item the flow $\varphi$ of \eqref{eq:lyapode} is defined on all $\mathbb{R}_+ \times K$,
    \item the solution remains in $K$, $\varphi(\mathbb{R}_+ \times K) \subset K$,
%     \item for all $\delta > 0$, there exists $\eta >0$ such that for all $\epsilon > 0$, there exists $T > 0$ such that for all $t \geq T$,
% \[
% \varphi^t(K \setminus \mathcal{L}^\delta) \subset \mathcal{L}^\epsilon \setminus \mathcal{M}^\eta_{++},
% \]
    % PREVIOUS \item and if $\mathcal{L} \cap K$ contains only local minima and is simple, for all $\epsilon > 0$, there exists $T > 0$ such that for all $t \geq T$,
    \item and if further the jail is contracting, 
% \[
%     \mathcal{C} \cap K = \argmin_K f,
% \] %is path-connected, 
    then solutions converge uniformly (in value) to the minimum: for all $\epsilon > 0$, there exists $T > 0$ such that for all $t \geq T$ and $x \in K$,
\[
    f(\varphi^t(x)) - \min_K f \leq \epsilon.
\]
% \[
% \varphi^t(K) \subset \mathcal{L}^\epsilon \cap K.
% \]
\end{enumerate}
\end{proposition}

The last point simply tells us that all points converge uniformly to $\argmin_K f$. We also prove that all points of $U$ converge to $\mathcal{C}$ in general (without assuming $K$ is contracting), but we lose the uniform convergence so necessary to the remainder of our structural results. A solution starting arbitrarily close to a strict maximum would take arbitrarily long to escape the strict maximum's vicinity.

\smallskip
 
Let $\psi^{s,t_0}_0(x)$ denote the solution of \eqref{eq:odescaledflat} at time $s$, starting from the initial condition $x(t_0) = x$, and let $\mathcal{C}_0(t_0)$ be the set of critical points of $f_{t_0}$ on the manifold $\mathcal{M}_0(t_0)$. If we apply the Lyapunov argument of proposition \ref{prop:lyap} to \eqref{eq:odescaledflat}, we get the following result: 

\begin{proposition}
\label{prop:lyapflat} %path-connected 
Let $K$ be a jail 
%n isolated level set 
of $f_{t_0}$ for some $t_0$, then, 
\begin{enumerate}
    \item the flow $\psi^{t_0}_0$ of \eqref{eq:odescaledflat} is defined on all $\mathbb{R}_+ \times K$,
    \item the solution remains in $K$, $\psi^{t_0}_0(\mathbb{R}_+ \times K) \subset K$,
    \item and if further $K$ is contracting, that is,
\[
    \mathcal{C}_0(t_0) \cap K = \argmin_K f_{t_0},
\]
then for all $\epsilon > 0$, there exists $\bar{s} > 0$ such that for all $s \geq \bar{s}$ and $x \in K$,
\[
    f(\psi^{s,t_0}_0(x),t_0) - \min_K f_{t_0} \leq \epsilon.
\]
\end{enumerate}
\end{proposition}

\begin{proof}
By hypothesis, $f_{t_0}$ is coercive but also, for all $x \in \mathbb{R}^n$,
\[
    - \nabla_x f(x, t_0)^\top P \nabla_x f(x, t_0) \leq 0,
\]
with equality if and only if $\nabla_x f(x,t_0) \in (\ker J)^\perp$, namely if and only if $x$ is critical at time $t_0$, so that $f_{t_0}$ is a Lyapunov function for \eqref{eq:odescaledflat} (at time $t_0$).
\end{proof}

%In the next section, we will make this statement uniform in $t_0$, it will allow the tracking of a local solution.

\subsection{Tracking Pockets of Local Minima}

%In all generality, we cannot expect a simple ODE to converge to an all-time global optimum of \eqref{eq:generalpb}, thus we naturally focus on local minima. 
To benchmark our solution's value, we will compare it to the value of a nearby local minima ``pocket''. For simplicity, one could assume the set of critical points of $f_t$ under the constraint $h_t = 0$ contains only strict maxima and minima, along with isolated saddle points, but this assumption would be wrong for two reasons. First, we might encounter plateaux in many practical applications. Second, even when considering solely strict local minima, their uniqueness is  meaningful outside of bifurcations:
%only on open time intervals, as bifurcations might occur elsewhere: 

\begin{example}
Consider for instance,
\[
    f(x,t) = x^4 - 2t x^2,
\]
with no constraints. Its minima are $\pm \sqrt{t}$ for $t \geq 0$, and $0$ for $t \leq 0$. If we were to follow $0$ in the opposite direction of time, the uniquness of the minimum would be lost at $t = 0$. Additionally, close enough to time $t=0$, the basins of attraction of each minimum do get arbitrarily close, so much so that any numerical error could ``flip'' the trajectory to the other minimum.
\end{example}

Well outside of the bifurcation, however, the local minima are identifiable and distinct enough for computational purposes. Therefore, to evaluate our ODE solution against the solution to the original problem \eqref{eq:generalpb}, we will place ourselves outside of bifurcations, keeping in mind that they are relatively isolated.
\bigskip

We would like to extend in time the definition of a jail, to define a prison. However not only should a prison be a jail at all time, but it should also vary ``continuously'' in time. To this end we introduce the notion of basin, which helps in the construction of a prison.

\begin{definition}[Slicing]
Let $X$ be a subset of some $A \times B$, we define the slicing of $X$ as the function,
\[
    \mathrm{s}X : a \in A \mapsto \{b \in B, ~(a,b) \in X\}.
\]
\end{definition}

\begin{definition}
We say that a correspondence $\phi : A \tto B$ is locally bounded if for all $a \in A$, there exists a compact $D \subset B$ and neighbourhood $I$ of $a$, such that $\phi(I) \subset D$.
\end{definition}

\begin{definition}[Basin]
Let $I \subset \mathbb R_+$ be a finite open interval on which we want to isolate a local minimum. Let $V \subset I \times \mathbb{R}^n$ be an open set, to be understood as a neighbourhood of the solution we are tracking. We say that $V$ is a basin if its slicing $\mathrm sV$ is locally bounded, if it is continuous in the sense that $\overline{\mathrm sV}$ is upper hemicontinuous, and regular in the sense that at all time $t \in I$,
\[
\overline{\mathrm sV(t) \cap \mathcal M(t)} = \overline{\mathrm sV(t)} \cap \mathcal M(t).
%\partial (\overline{\mathrm sV(t)} \cap \mathcal{M}(t)) = \emptyset.
\]
\end{definition}

\begin{proposition}
\label{prop:Kcont}
If $V$ is a basin, the correspondence,
\[
    K: t \in I \mapsto \overline{\mathrm sV(t) \cap \mathcal M(t)},
\]
is continuous and has non-empty compact values.
\end{proposition}

The proof is presented in the Appendix. By the maximum theorem, we can then define the nonempty-compact-valued upper hemicontinuous correspondence,
\[
    K^*: t \in I \mapsto \argmin_{K(t)} f_t,
\]
along with the continuous,
\[
    f^*: t \in I \mapsto \min_{K(t)} f_t.
\]
The definition of $f^*$ will be a benchmark for our ODE solution. If at all time $t \in I$,
\[
    \mathrm sV(t) \cap \mathcal{C}_0(t) = K^*(t),% \subset \mathrm sV(t),
\]
we say that $V$ is contracting, essentially all the critical points of $V$ are global minima (in $V$).

\begin{proposition}
\label{prop:basin}
If $V$ is a basin over $I$ open, there exists $\epsilon >0$, $\hat f : I \to \mathbb R$ and $\hat K : I \tto \mathbb R^n$ with compact non-empty values, such that at all time $t \in I$ and $x \in K(t)$,
\[
    f(x,t) \leq \hat f(t) \iff x \in \hat K(t),
\]
while,
\[
    \hat K(t) \subset \mathrm sV(t) \cap \mathcal M(t),
\]
and,
\[
    \hat f (t) \geq f^*(t) + \epsilon.
\]
\end{proposition}

Automatically,  $\hat K(t)$ is compact as closed subset of $K(t)$ compact and non-empty as it contains $K^*(t) \neq \emptyset$. This proposition spares us from the assumption that a continuous correspondence of isolated level sets is somehow  given. This sets up a natural structure for proving the main result, much as how the isolated level set was the right structure to prove the convergence of $\psi_0^{t_0}$.

\subsection{The Main Result}

We introduce the flow $\varphi_\alpha$ of \eqref{eq:ode} with initial condition taken at time $t_0$, which is the original ODE of parameter $\alpha$, not time-rescaled. We are interested in (1) showing it is well defined, (2) showing it does track a local solution. 

\begin{definition}[$\epsilon$-tracking within $\nu$-time]
%TO BE COMPLETED
Let $V$ be a contracting basin on some open interval $I_0$, and $I = [t_0, t_1] \subset I_0$.
%Let $(K,\hat{K})$ be a prison 
%Let $V$ be a persistent basin 
%on $I = [t_0, t_1]$ closed interval, 
%and $J = [t_0, t_1] \subset I$ be a closed interval
Let $\epsilon > 0$ be a precision, $\nu >0$ be a reaction time, and $\alpha$ be a momentum parameter. We say that $\alpha$ allows for $\epsilon$-tracking within $\nu$-time when:
\begin{enumerate}
    \item the solution of \eqref{eq:ode}$, \varphi_\alpha$, is defined over $I \times \hat K(t_0)$,
    \item at all times $t \in I$, the solution remains in the basin,
\[
    \varphi^t_\alpha(\hat K(t_0)) \subset K(t),
\]
    \item and for all $t \in [t_0 + \nu, t_1]$ and $x \in \hat K(t_0)$, the solution is a good approximation of the minimum,
\[
    f(\varphi_\alpha^t(x),t) - f^*(t) \leq \epsilon.
\]
\end{enumerate}
\end{definition}

\begin{theorem}[Tracking with Equalities]
\label{thm:ode}
Consider the time-varying problem with equalities \eqref{eq:equalityproblem},
under Assumptions \ref{as:functionclass}, \ref{ass:feaseq} and \ref{ass:cont}. 
Let $V$ be a contracting basin on some open interval $I_0$, and $I = [t_0, t_1] \subset I_0$.
%Let $(K, \hat{K})$ be a prison on $I$ closed interval.
%Let $V$ be a persistent basin and $J = [t_0, t_1] \subset I$ be a closed interval. 
Then, for any precision $\epsilon > 0$ and any reaction time $\nu >0$, there exists $\bar{\alpha} > 0$ such that for all $\alpha \in (0, \bar{\alpha}]$,
$\alpha$ allows for $\epsilon$-tracking within $\nu$-time.
% if $x \in K(t_0)$ then for all $\alpha < \bar{\alpha}$ and $t \in [t_0+\nu, t_1]$, $\psi^{\nicefrac{t-t_0}{\alpha},t_0}_\alpha(x)$ is defined and, % and $f(x, t_0) \leq \argmin_{K(t_0)} f_{t_0} + \epsilon$, 
% \[
%     f(\psi^{\nicefrac{t-t_0}{\alpha},t_0}_\alpha(x),t) - f^*(t) \leq \epsilon.
% \]
\end{theorem}

The complete proof is available in the Appendix.
It relies on making previous statements global and applying them to \eqref{eq:odescaledparam}. 

Now, let us present the main result:

% \begin{assumption}[Feasibility w.r.t. Equalities and Inequalities]
% \label{ass:feaseq}
%  At all times $t$, there exists a point $x$ such that $h_i(x, t) = d_i(t)$ for all $i = 1, \ldots, p$ and $g_i(x, t) \le e_i(t)$ for all $j = 1 \ldots q$.
% \end{assumption}

\begin{theorem}[Tracking with Inequalities]
Consider the time-varying problem with inequalities \eqref{eq:inequalityproblem}, 
under Assumptions \ref{as:functionclassineq},  \ref{ass:feasineq}, and \ref{ass:contineq}.
Let $V$ be a contracting basin on some open interval $I_0$, and $I = [t_0, t_1] \subset I_0$.
%Let $V$ be a persistent basin and $J = [t_0, t_1] \subset I$ be a closed interval. 
Then, for any precision $\epsilon > 0$ and any reaction time $\nu >0$, there exists $\bar{\alpha} > 0$ such that for all $\alpha \in (0, \bar{\alpha}]$,
$\alpha$ allows for $\epsilon$-tracking within $\nu$-time.
%for inequalities and discontinuous right-hand sides. 
\end{theorem}

\begin{proof}
Note that the time-varying optimization program with inequalities \eqref{eq:inequalityproblem} can be reformulated using slack variables $z \in \mathbb{R}^q$:
\begin{align}
\label{eq:equalityproblem}
	\inf_{x\in \mathbb{R}^n, \ z \in \mathbb{R}^q} &  ~ f(x,t)\\ 
	\notag
	\text{ s.t. } & h_i(x) = d_i(t), \quad i = 1,\hdots,p \\
	\notag 
	& g_j(x) + z_j^2 = e_j(t), \quad j = 1,\hdots,q.
\end{align}
It is well-known \cite[Section 3.2.2]{Bertsekas/99} that for $x \in \mathbb{R}^n$, $x$ is a local solution of \eqref{eq:inequalityproblem} if and only if there exists $z \in \mathbb{R}^q$ such that $[x~z]^\top$ is a solution of \eqref{eq:equalityproblem}. 
The structural results on  \eqref{eq:equalityproblem} can thereby get extended to problems with inequalities.
\end{proof}

The nature of this result is global in time as long as we are away from bifurcations and the isolated time-discontinuities which may occur. In turn, if the problem is discontinuous at isolated times, we may apply our result in between them.

\section{Algorithms}

Algorithm \ref{alg:ONCCO}
presents the schema of our algorithmic approach. It has two crucial components:
\begin{itemize}
\item a reformulation of our continuous-time model so as to avoid the use of matrix inversion, 
%\item a patching procedure, which upon detection of a discontinuity in the right-hand side applies a refinement step,
\item the ODE solver, which can be seen as an predictor-corrector method. 
\end{itemize} 
Let us introduce them in turn.

\subsection{A Reformulation}

Any model we consider, e.g.  \eqref{eq:odeineq}, is based on a system of first-order optimality conditions, which are a system of linear or non-linear equations, i.e. in the static case:
\begin{align}
\label{Rammoriginal}
F(u) = 0
\end{align}
for some spaces $X, Y$ and a map $F: X \to Y$. Let us denote the derivative of $F$ by $F'$. We wish to solve the system of equations \eqref{Rammoriginal}, which could be done using the continuous-time Newton method:
\begin{align}
\label{RammNewton}
\dot u &= -[F'(u)]^{-1} F(u)  \\
 u(0) &= u_0 \notag
\end{align}
for finding the zeros of $F(u)$. This, however, involves matrix inversion, which is computationally expensive. We would like to use, again in the static case,
\[
    F(x) = P \nabla f(x).
\]
The computation of $[F'(x)]^{-1}$ involves the inverse of the Hessian $\nabla^2 f(x)$. But even for a sparse Hessian, the inverse can be fully dense. 

%Instead, we consider a higher-dimensional problem in $(u, Q)$:
%\begin{align}
%\label{Rammapprox}
%    \dot u & = - Q[F(u) - f] \\
%    u(0) & = u_0 \notag \\
%    \dot Q & = -[(F'(u))^* F'(u)] Q + (F'(u))^* \notag  \\
%    Q(0) & = Q_0 \notag 
%\end{align}
%This approach may seem heuristic at first, but in the Appendix, % \ref{appendix-RammTheorem}, 
%we 
%show that under mild assumptions 
%of bounded Frechet derivatives of $F$,
% $u_0$ sufficiently close to $y$,
%and $Q_0$ sufficiently close to $(F'(y))^{-1}$, the lifted problem  \eqref{Rammapprox} has a unique global solution
%and there exists a time $t$ and $u(t)$ such that
%$u(t) = y$, i.e., 
%\begin{align}
%    \lim_{t \to \infty} \| u(t) - y \| = 0 \\
%    \lim_{t \to \infty} \| Q(t) - (F'(y))^{-1} \| = 0.
%\end{align}
%This justifies the use of the much simpler matrix-vector multiplication, instead of matrix inversion.
%\textbf{Alternative version, with vector multiplication}

Instead, we can consider a higher-dimensional problem in $(u, v)$:
\begin{align}
\label{Rammapprox}
    \dot u & = - v \\
    u(0) & = u_0 \notag \\
    \dot v & = \rho [F(u) - F'(u) v] \notag  \\
    v(0) & = v_0 \notag 
\end{align}
where $\rho$ is a scalar constant.
%to be tuned.
This \eqref{Rammapprox} may seem heuristic at first, but as presented in the appendix, % \ref{appendix-RammTheorem}, 
one 
shows that under mild assumptions, including bounded Fr\'echet derivatives of $F$,
the lifted problem \eqref{Rammapprox} has a unique global solution, 
%and there exists a time $t$ and $u(t)$ such that
%$u(t) = y$,
i.e., 
\begin{align}
    u(t) \to_t x \\
    v(t) \to_t (F'(x))^{-1}F(x),
\end{align}
for some $x \in X$ root of $F$.

This justifies the use of the much simpler matrix-vector multiplication, instead of matrix inversion in the time-invariant case. We will see how to adapt this method to our on-line problem.

%\subsection{A Patching Procedure}
%\label{sec:refinement}

\subsection{Solvers for the ODE}

There are a variety of methods for solving the ODE \eqref{Rammapprox}. A Runge-Kutta of order four is a common choice considering:
\begin{align}
\label{eq:corrector}
u_{n+1}&=u_{n}+{\tfrac {1}{6}}\left(k_{1}+2k_{2}+2k_{3}+k_{4}\right),
\end{align}
%for n = 0, 1, 2, 3, ..., using[2]
wherein the predictors for a step-size $h > 0$ are:
\begin{align}
\label{eq:predictor1}
k_{1}&=h\ F(t_{n}, u_{n}),\\
k_{2}&=h\ F\left(t_{n}+{\frac {h}{2}}, u_{n}+{\frac {k_{1}}{2}}\right),\\
k_{3}&=h\ F\left(t_{n}+{\frac {h}{2}},u_{n}+{\frac {k_{2}}{2}}\right),\\
k_{4}&=h\ F\left(t_{n}+h, u_{n}+k_{3}\right).
\label{eq:predictor4}
\end{align}

With the fourth-order Runge-Kutta method, each iteration hence requires 4 evaluations of $F(u)$ and four evaluations of the gradient $F'(u)$ and a minimal amount of arithmetic.

\subsection{Our Algorithm}

We rewrite our ODE to avoid any matrix inversion or multiplication. We propose the following ODE as a replacement of \eqref{eq:ode},
\begin{align}
    \dot x =&~ - \frac{1}{\alpha} \nabla_x f(x(t),t) + \frac{1}{\alpha} J^\top u - J^\top v \label{eq:odealgo} \tag{ODE$^\circ$} \\
    \dot u =&~ \rho (J \nabla_x f(x(t), t) - J J^\top u) \notag \\
    \dot v =&~ \rho (h'(x(t),t) - J J^\top v) \notag
\end{align}
%Note $J = \mathcal{J}_{h_t}(x(t))$. 
In the limit of $\rho > 0$ going to infinity, with a time rescaling $s = \rho(t-t_0)$, we find the limit ODE in $s$,
\begin{align}
    \dot x =&~ 0 \label{eq:odealgoflat} \tag{ODE$^\circ_0$} \\
    \dot u =&~ J \nabla_x f(x(t_0), t) - J J^\top u \notag \\
    \dot v =&~ h'(x(t_0),t_0) - J J^\top v \notag
\end{align}
Since $JJ^\top \succ 0$, we have, with $J$ is taken at $(x(t_0), t_0)$,
\begin{align}
    u(s) &\to_s (J J^\top)^{-1} J \nabla_x f(x(t_0), t_0) \\
    v(s) &\to_s (J J^\top)^{-1} h'(x(t_0),t_0).
\end{align}
%By a similar reasoning as our tracking argument, 
We can show that for $\rho$ great enough the solution $x$ of \eqref{eq:odealgo} approximates (in objective) the solution of \eqref{eq:ode}.

For simplicity and practicality, we may dissociate updates of $u,v$ from those of $x$. We run the update of $u,v$ until a certain acceptable error margin is reached, then update $x$ by one step. The cost per update of $x$ amounts to roughly $20$ computations of $J$ (depending on the number of updates of $u,v$) and for the rest, some matrix vector multiplications.

% \begin{algorithm}[tb!]
%   \caption{onnccoo: On-Line Non-Convex Constrained Optimization}
%   \label{alg:ONCCO}
% \begin{algorithmic}[1]
%   \STATE {\bfseries Input:} Initial point. Threshold $\theta$. 
%           Updates  $\{F_n\}_0^{\infty}$, 
%           available sequentially at times $\{t_n\}_0^{\infty}$. \\
%   \STATE Initialise time $t_0$ to the current time. 
%   \FOR {$n = 1$ {\bfseries to}  $\infty$}
%   \STATE  If the difference between the right-hand sides (or $F_n$ and $F_{n-1}$ in some suitable norm) is larger than $\theta$, run a
%     refinement step %, as explained in Section~\ref{sec:refinement}. 
%   \STATE  Perform the predictor step, e.g., using (\ref{eq:predictor1}--\ref{eq:predictor4}) for the fourth-order Runge-Kutta method
%   \STATE  Perform the corrector step, e.g., using \eqref{eq:corrector} for the fourth-order Runge-Kutta method
%   \STATE  Update time $t_{n+1}$ to current time; update time step $h = t_{n+1} - t_{n}$.
%   \ENDFOR
% \end{algorithmic}
% \end{algorithm}

\begin{algorithm}[tb!]
   \caption{onnccoo: On-Line Non-Convex Constrained Optimization}
   \label{alg:ONCCO}
\begin{algorithmic}[1]
   \STATE {\bfseries Input:} Initial point close to global optimum, threshold $\theta$, small momentum $\alpha$.
%          Updates  $\{F_n\}_0^{\infty}$, 
 %         available sequentially at times $\{t_n\}_0^{\infty}$. \\
  % \STATE Initialise time $t_0$ to the current time. 
   \FOR {$n = 1$ {\bfseries to}  $\infty$}
   \STATE  Update $u,v$ until threshold $\theta$ is reached, using the predictor-corrector method, e.g., using (\ref{eq:predictor1}--\ref{eq:predictor4}) for the fourth-order Runge-Kutta method%, as explained in Section~\ref{sec:refinement}. 
   \STATE  Update $x$ of one step using the predictor-corrector method, e.g., using (\ref{eq:predictor1}--\ref{eq:predictor4}) for the fourth-order Runge-Kutta method
%   \STATE  Perform the corrector step, e.g., using \eqref{eq:corrector} for the fourth-order Runge-Kutta method
%   \STATE  Update time $t_{n+1}$ to current time; update time step $h = t_{n+1} - t_{n}$.
   \ENDFOR
\end{algorithmic}
\end{algorithm}

\begin{figure}[t!]
%\begin{minipage}{0.99\textwidth}%
        \includegraphics[page=28, width=0.99\textwidth,clip]{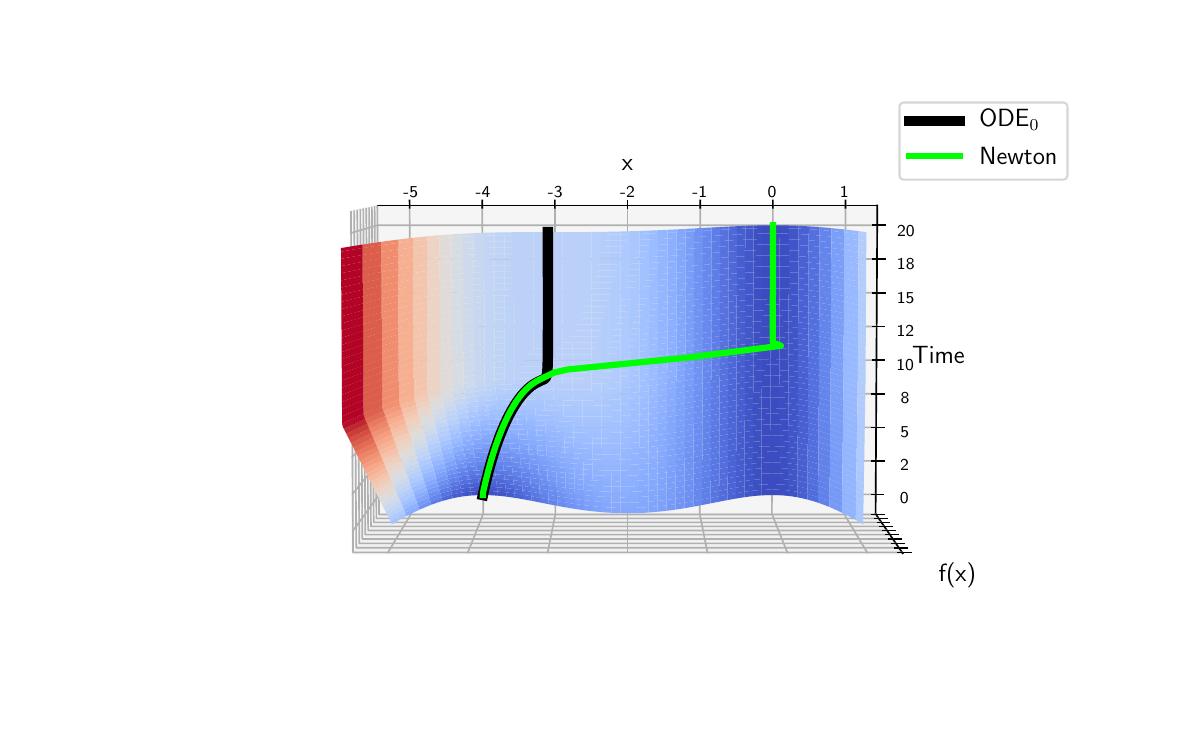}
\caption{The trajectories of ODE$_0$ and the continuous Newton method starting from the same initial point $x_0 = -4$ on \eqref{ex:simple}.}\label{fig1right}
%\end{minipage}
%\quad
%\begin{minipage}{0.49\textwidth}%
%    \centering
%\includegraphics[page=8, %width=\textwidth,clip]{examples}  %  
%    \caption{A rotation of Figure %\ref{fig1right}.
%    }\label{fig1left}
%\end{minipage}
%\quad
%\quad
\vspace{-0.5em}
\end{figure}
\subsection{A Computational Illustration}
\label{sec:exp}

Consider a very simple univariate time-varying function:

\begin{example}
\begin{align}
f(x, t) :=  x^4 & + 8x^3 + 18x^2 \label{ex:simple}  \\
+ & \begin{cases}
- 2x^2 (10 -t)/10 & \textrm{ if } t \le 10 \\
0 & \textrm{ otherwise}
\end{cases}
\notag 
\end{align}
\end{example}

Figure \ref{fig1right} %--\ref{fig1left} 
presents the evolution of ODE$_0$ and the continuous Newton method over time, starting at $x_0 = -4$.
Clearly, while the ODE$_0$ does not escape a saddle point at $x = -3$, Newton method does reach an optimum at $x_T = 0$.

\section{Conclusions} % and Related Work}

We have presented some of the first structural results for on-line non-convex optimization. Based on these results, we have derived an algorithm, whose per-iteration complexity is substantially lower than per-iteration complexity of a gradient method on the constraint non-convex problem, but whose use allows for tracking of local optima, under mild assumptions. 

Further work should include tuning of the various constants intervening in the algorithm, and implementing the algorithm on a large instance. We hope our first results will spur further research in on-line non-convex constrained optimization.

%We could try to derive
%We could consider replacing the matrix inversion by (8) of 
% https://www.ist.uni-stuttgart.de/de/forschung/pdfs./Newton_ES.pdf
%\cite{LABAR2019356}. This simplifies many steps of the proofs.

%\clearpage

\bibliography{refs,completion,pursuit,time-varying,calculus}

\appendix
\newpage 

\newpage

\section{A note on jails, and the Lyapunov argument}

\begin{lemma}
If $A,B$ are two disjoint closed sets, then there exists an open $V$ such that $A \subset V$ and $\bar{V} \cap B = \emptyset$.
\end{lemma}

\begin{proof}
For $a \in A$, there exists an open neighbourhood $V_a$ of $a$ that does not intersect $B$. Indeed otherwise, by closure of $B$, $a \in B$ so that $A \cap B \neq \emptyset$. Then,
\[
    V = \bigcup_{a \in A} V_a,
\]
is open as union of open sets, contains $A$ yet does not intersect $B$.
\end{proof}

\iffalse
$K$ contracting jail means,
\[
    \mathcal C \cap K = \argmin_{K} f.
\]

\begin{lemma}
If $V$ is a jail, then there exists an open $W \subset V$ such that,
\[
    \argmin_{M \cap \bar{W}} f = \mathcal{C} \cap \bar{W} \subset W.
\]
\end{lemma}

\begin{proof}
As $\mathcal{C} \cap \bar{V}$ and $\mathbb{R}^n \setminus V$ are two disjoint closed sets, the previous lemma gives us an open $W$ such that $\bar{W} \subset V$ and $\mathcal{C} \cap \bar{V} \subset W$. Note that always,
\[
    \argmin_{M \cap \bar{W}} f \subset \mathcal{C} \cap \bar{W},
\]
with what we had previously,
\[
    \argmin_{M \cap \bar{W}} f \subset \mathcal{C} \cap \bar{W} \subset \mathcal{C} \cap \bar{V} \subset W.
\]
But also, all points of,
\[
    \argmin_{M \cap \bar{V}} f = \mathcal{C} \cap \bar{V},
\]
have the same value by $f$, as a result, all points of,
\[
    \mathcal{C} \cap \bar{W} \supset \argmin_{M \cap \bar{W}} f,
\]
have the same value by $f$, which then must be,
\[
    \min_{M \cap \bar{W}} f,
\]
thus all points of $\mathcal{C} \cap \bar{W}$ are in $\argmin_{M \cap \bar{W}} f$.
\end{proof}
\fi

\begin{proof}[Proof of Lemma \ref{lem:lowerlevel}]
For $c \in \mathbb{R}$, surely $(f,h)^{-1}((-\infty, c] \times \{0\})$ is closed as continuous preimage of a closed set. If it were unbounded, there would exist $(x_n)_{n \in \mathbb{N}}$ unbounded such that $f(x_n) \leq c$, contradicting the coercivity of $f$, thus the lower level set (under constraint $h=0$) $(f,h)^{-1}((-\infty, c] \times \{0\})$ is compact.
\end{proof}

From the first lemma, if $K$ is a jail (with level set $L$), then as $K$ and $L \setminus K$ are disjoint closed sets, there exists an open $V$ such that $K \subset V$ while $\bar V \cap (L \setminus K) = \emptyset$, and thus $\bar V \cap L = K$. We can now simply invoke such $V$ given a jail $K$. From the above lemma, we see that $L$ is compact, thus so is $K$ as soon as it is closed.

To prove proposition \ref{prop:lyap} we need to first introduce couple standard lemmas on dynamical systems.

\begin{lemma}[Escape dilemma]
\label{lem:escape}
Let $K \subset \mathbb{R}^n$ be nonempty compact. Let $g : \mathbb{R}^n \to \mathbb{R}^n$ be locally Lipschitz continuous and $x_0 \in K$. Consider the following Cauchy problem,
\begin{equation}
    \label{eq:generalode1}
    \dot{x} = g(x), ~x(0) = x_0.
\end{equation}
By Picard-Lindel\"of theorem, there exists a unique maximal solution to \eqref{eq:generalode1}, $(I,x)$. We refer to $I_+ \triangleq I \cap \mathbb{R}_+$. Then, either $\sup I = \infty$, the solution is eternal, or $x(I_+) \subsetneq K$, the solution leaves $K$.
\end{lemma}

\begin{proof}[Proof of Lemma \ref{lem:escape}]
Assume that $\sup I = a < \infty$ and $x(I_+) \subset K$. We will prolong $(I,x)$, contradicting its maximality.
\medskip

Let $(a_n)_{n \in \mathbb{N}}$ be a sequence of $I$ such that $a_n \to_n a$, by compactness of $K$, there must be a subsequence of $(a_n)$, say $(b_n)_{n \in \mathbb{N}}$, such that $x(b_n) \to_n \tilde{x} \in K$ some limit. Then for any sequence $(c_n)_{n \in \mathbb{N}}$ of $I$ such that $c_n \to_n a$,
\begin{align*}
    \| x(c_n) - x(a) \|
    & \leq \| x(c_n) - x(b_n) \| + \| x(b_n) - x(a) \| \\
    & \leq \max_K \| g \| |c_n-b_n| + \| x(b_n) - x(a) \| \\
    & \to_n 0,
\end{align*}
so that,
\[
    x \to_a \tilde{x}.
\]
Define now $(J, y)$ the maximal solution of \eqref{eq:odescaledflat} with initial condition $y(a) = \tilde{x}$, and,
\[
    z : t \in I \cup J \mapsto \begin{cases}
    x(t) &\text{if } t \in I, \\
    y(t) &\text{otherwise.}
    \end{cases}
\]
Surely $(I \cup J, z)$ is a solution to \eqref{eq:odescaledflat} with initial condition $z(0) = \bar{x}$, which \emph{strictly} prolongs $(I,x)$, so that $(I,x)$ is not maximal, which is absurd. We deduce that the solution is defined for all positive times.
\end{proof}

\iffalse

\begin{proof}[Proof of Proposition \ref{prop:basin}]
% Let $W$ be an open isolating $K$ from $L$, following the definition of isolated level set. Let $W'$ be an open set such that $K \subset W'$ and $(\mathcal{C} \setminus K) \cap W' = \emptyset$. Then,
% \[
% V = \{ x \in W \cap W', ~ d(x, K) < 1\},
% \]
% is a basin.

Let,
\[
    \hat{f} = \frac{1}{2} \left(
    \min_{M \cap \bar{V}} f +
    \min_{M \cap \partial V} f
    \right) > \min_{M \cap \bar{V}} f,
\]
as otherwise we would have,
\[
\min_{M \cap \partial V} f = \min_{M \cap \bar{V}} f,
\]
thus,
\[
    (M \cap \partial V) \cap \argmin_{M \cap \bar{V}} f \neq \emptyset,
\]
contradicting the hypothesis that,
\[
    \argmin_{M \cap \bar{V}} f \subset V.
\]

Define the lower level set,
\[
    \hat{L} = (f,h)^{-1}((-\infty, \hat{f}] \times \{0\}),
\]
and the set,
\[
    \hat{K} = \hat{L} \cap V.
\]
Let $(x_n)$ be a sequence in $\hat{K}$, converging to some $x \in \hat{L} \cap \bar{V}$, closed. If $x \in \partial V$ we have,
\[
f(x) \geq \min_{M \cap \partial V} f > \hat{f},
\]
which is absurd, thus $x \in V$. As a result $\hat{K}$ is closed, bounded as subset of $V$ bounded, thus is compact.

therefore is an isolated level set, of value $\hat{f} > \min_{M \cap \bar{V}} f$, it is a jail such that,

By definition of $\hat{K}$,
\[
    \argmin_{M \cap \bar{V}} f = (f^{-1},h)((-\infty, f^*] \times \{0\}) \cap   \subset K \subset V,
\]
and,
\[
    \mathcal{C} \cap K = \argmin_K f.
\]
\end{proof}
\fi

\begin{proof}[Proof of Proposition \ref{prop:lyap}]
The proofs consists in two parts. First, realising that $f$ decreases along solutions we conclude that a solution starting in $K$ cannot leave $K$. With the escape dilemma, this shows that the solution exists for all positive times. Second, either $x(0)$ is already close enough to the minimum, or it belongs to a set where $f$ decreases steadily, $x$ cannot stay in this set and thus gets close to the minimum.

\proofpart{First two points}
The first point is a direct consequence of the escape dilemma (lemma \ref{lem:escape}). Let 
%$c \in \mathbb{R}$ be the edge of
$K$ be a jail 
%$K$ be a subset 
of a lower level set $L$, isolated by the open set $V$ (namely $L \cap V = K$). 
% such that $K$ is a path-connected component of the lower lever set $L \triangleq (f,h)^{-1}((-\infty,c] \times \{0\})$.
Let $x_0 \in K$ and $(I,x)$ be a maximal solution of \eqref{eq:lyapode} with initial condition $x(0) = x_0$. Surely,
\begin{align*}
        \frac{\mathrm{d}}{\mathrm{d} s} f(x(s)) = \nabla f(x(s))^\top g(x(s)) &\leq 0,\\
        \frac{\mathrm{d}}{\mathrm{d} s} h(x(s)) = \nabla h(x(s))^\top g(x(s)) &= 0,
\end{align*}
thus $f \circ x$ decreases while $h \circ x = 0$, hence $x$ cannot leave $L$ compact (lemma \ref{lem:lowerlevel}) in positive times, thus must be defined for all positive times (lemma \ref{lem:escape}).

Suppose $x$ were to leave $K$, and define then,
\[
t_0 = \inf_{t \geq 0, ~x(t) \not \in K} t.
\]
By continuity of $x$ and closedness of $K$, we have $x(t_0) \in K$, thus there exits $t_1 > t_0$ such that $x(t_1) \in V \setminus K$. But as $L \cap (V \setminus K) = \emptyset$, $x(t_1) \not \in L$, this is absurd. Therefore, $x$ remains in $K$. We call $\varphi$ the flow of~\eqref{eq:lyapode}.

\proofpart{Last point}
%On the last point now. 
Define,
\[
    K^* \triangleq \argmin_K f \textrm{ and } f^* \triangleq \min_K f.
\]
Call $c$ the edge of $K$ and let $\epsilon > 0$ be such that $f^* + \epsilon < c$. We define,
\[
K^\epsilon \triangleq \{x \in K, ~f(x)-f^* < \epsilon\},%, h(x) = 0\},
\]
so that $K^* \subsetneq K^\epsilon \subsetneq K$. Define,
\[
    \lambda = - \max_{K \setminus K^\epsilon} \nabla f^\top g > 0,
\]
positive since there is no critical point on the compact $K \setminus K^\epsilon$, as,
\[
(K \setminus K^\epsilon) \cap \mathcal{C} = (K \cap \mathcal{C}) \setminus K^\epsilon = K^* \setminus K^\epsilon = \emptyset.
\]
Define also,
\[
    R = \max_K f - f^* \geq 0,
\]
and the time,
\[
    T = \nicefrac{R}{\lambda}.
\]

Let $x_0 \in K^\epsilon \subset K$, we know then that $\varphi^t(x_0)$ exists and belong to $K$ for all $t \geq 0$. Furthermore, $f \circ \varphi(x_0)$ is nonincreasing, thus $f(\varphi^t(x_0)) < f^* + \epsilon$ for all $t \geq 0$.

Let now $x_0 \in K \setminus K^\epsilon \subset K$ compact, we know then that $\varphi^t(x_0)$ exists and belong to $K$ for all $t \geq 0$. We also know from the previous argument that if $f(\varphi^t(x_0)) < f^* + \epsilon$ for some $t \geq 0$, then $f(\varphi^s(x_0)) < f^* + \epsilon$ for all $s \geq t$. But there must exist such a $t \in [0, T]$, otherwise,
\begin{align*}
    f(\varphi^T(x_0)) &= f(x_0) + \int_0^T \nabla f^\top g (\varphi^t(x_0)) \mathrm dt \\
    &\leq \max_K f - T \lambda = f^*,
\end{align*}
which is a contradiction.

As a result, no matter $x_0 \in K$, for all $t \geq T$, $\varphi^t(x_0) \in K^\epsilon$, or shorter, for all $t \geq T$, $\varphi^t(K) \subset K^\epsilon$.
% Let $x \in K^\epsilon \subset K$, define the lower level set,
% \[
% L_x = (f,h)^{-1}((-\infty, f(x)] \times \{0\}) \subset L,
% \]
% and the following jail, 
% \[
%     K_x = \{y \in K, ~f(y) \leq f(x)\} = K \cap L_x.
% \]
% Clearly it is a compact subset of $L_x$. Let $V$ be an open neighbourhood of $K$ such that $L \cap V = K$. Surely $L_x \cap V \subset L \cap V = K$, thus,
% \[
% L_x \cap V = L_x \cap V \cap K = L_x \cap K = K_x,
% \]
% so that $V$ isolates $K_x$ from $L_x$, $K_x$ is indeed a jail. Applying the second point of proposition \ref{prop:lyap}, which has been proven earlier, for all $t \geq 0$,
% % while,
% % \[
% % L_x \cap V = L_x \cap V \cap K = V \cap K_x \subset K_x.
% % \]
% % . Consider the following continuous function,
% % \[
% %     y \in K^\epsilon \mapsto d(y, \overline{L \setminus K}),
% % \]
% %since it is a constrained lower level set, for all $t \geq 0$,
% \[
%     \varphi^t(x) \in K_x \subset K^\epsilon.
% \]
% If $x \in K \setminus K^\epsilon$, it is clear that there must exist $0 < t \leq \nicefrac{R}{\lambda}$ such that,
% \[
%     \varphi^t(x) \in K^\epsilon,
% \]
% after which the previous conclusion states it remains in $K^\epsilon$. Indeed, otherwise we would have,
% \[
% R - \epsilon \geq f(0) - f(\nicefrac{R}{\lambda}) \geq R.
% \]
% As a result, naming $T = \nicefrac{R}{\lambda}$, for all $t \geq T$,
% \[
%     \varphi^t(K) \subset K^\epsilon. \qedhere
% \]
\end{proof}

\clearpage

\section{On the continuity of $K$ and existence of $\hat K$}

\subsection{Lower hemicontinuity of $K$}

\begin{lemma}
Let $h : U \times \mathbb{R}^n \to \mathbb{R}^m$ be a continuously differentiable function, where $U$ is an open of $\mathbb{R}^p$, to be understood as the space of parameters. Denote $h_u = h(u, \cdot)$. Assume that for all $(u,x) \in U \times \mathbb{R}^n$ such that $h(u,x) = 0$, the partial differential with respect to the second variable,
\[
y \mapsto \mathrm{d}h(u,x)(0,y) = \mathrm{d}h_u(x)(y),
\]
is surjective, namely all roots of $h$ are regular. Define,
\[
\mathcal{M} : u \in U \mapsto h_u^{-1}(\{0\}) = \{x \in \mathbb{R}^n, ~h(u,x)=0\},
\]
then, $\mathcal{M}$ is lower hemicontinuous.
\end{lemma}

\begin{proof}
Let $u_0 \in U$ and $V$ be an open of $\mathbb{R}^n$ intersecting $\mathcal{M}(u_0)$. Let then $x_0 \in \mathbb{R}^n$ such that $h(u_0,x_0) = 0$ and $(u_0, x_0) \in V$. We will prove that there exists an open neighborhood $W$ of $u_0$ such that $\mathcal{M}(u) \cap V \neq \emptyset$ for all $u \in W$.

% For clarity, we define the following surjective (by hypothesis) map,
% \[
% \lambda : x \in \mathbb{R}^n \mapsto \mathrm{d}h(u_0, x_0)(0, x).
% \]
We note $S$ a supplement of $\ker \mathrm{d}h_{u_0}$, namely a subspace of $\mathbb{R}^n$ such that $S \oplus \ker \mathrm{d}h_{u_0} = \mathbb{R}^n$. We let $(f_1, \dots, f_p)$ be a basis of $S$ and $(f_{p+1}, \dots, f_n)$ be a base of $\ker \mathrm{d}h_{u_0}$, while $(e_1, \dots, e_n)$ is the canonical basis of $\mathbb{R}^n$. Consider now the change of basis map,
\[
\Phi : e_i \mapsto f_i,
\]
and define for $s \in \mathbb{R}^m, k \in \mathbb{R}^{n-m}, u \in U$,
\[
\tilde{h}(u,s,k) = h(u,\Phi(s,k)).
\]

Defining $(s_0, k_0) = \Phi^{-1}(x_0)$, we have for all $(s,k) \in \mathbb{R}^n, u \in \mathbb{R}^p$,
\begin{align*}
\mathrm{d}\tilde{h}(u_0,s_0,k_0)(u,s,k) &= \mathrm{d}h(u_0,x_0)(u, \mathrm{d}\Phi(s_0, k_0)(s,k)) \\
&= \mathrm{d}h(u_0,x_0)(u, \Phi(s,k)),
\end{align*}
so that,
\[
\tilde{\lambda} : s \in \mathbb{R}^p \mapsto \mathrm{d}\tilde{h}(u_0,s_0,k_0)(0,s,0) = \mathrm{d}h(u_0,x_0)(0,\Phi(s,0)),
\]
is invertible. Applying the implicit function theorem, there exists an open neighborhood $\Omega$ of $(u_0,x_0)$, a neighborhood $W$ of $(u_0, k_0)$ and a continuously differentiable function $\phi$ defined on $W$ to $\mathbb{R}^p$, such that,
\begin{align*}
&~ ((u,s,k) \in \Omega \text{ and } \tilde{h}(u,s,k)) \\
\iff &~
((u,k) \in W \text{ and } s = \phi(u,k)).
\end{align*}
Without loss of generality, $W$ is a ball centered on $(u_0, k_0)$. We call $\tilde{W} = W \cap (U \times \{0\})$, and see that for all $u \in \tilde{W}$ open neighborhood of $u_0$,
\[
h(u,\Phi(\phi(u,0),0)) = \tilde{h}(u,\phi(u,0),0) = 0.
\]
Without generality, we can assume that $\Phi(\phi(W,0),0) \subset V$ (by continuity of $\phi$ and $\Phi$), therefore $\mathcal{M}$ is lower hemicontinuous.
\end{proof}

\begin{lemma}
If $\phi : A \tto B$ is a lower hemicontinuous function, then so is $\bar{\phi}$.
\end{lemma}

\begin{proof}
Let $a_0 \in A$ and $V \subset B$ be an open set intersecting $\bar{\phi}(a_0)$. Let then $b_0 \in \bar{\phi}(a_0) \cap V$. Let $\epsilon > 0$, such that $\mathcal{B}(b_0, \epsilon) \subset V$, we note that this open ball also intersects $\phi(a_0)$ as otherwise $b_0$ would not belong to the closure of $\phi(a_0)$.

By lower hemicontinuity of $\phi$, there exists an open neighbourhood $I$ of $a_0$ such that for all $a \in I$, $\phi(a) \cap \mathcal{B}(b_0, \epsilon) \neq \emptyset$. \emph{A fortiori}, for all $a \in I$,
\[
    \bar{\phi}(a) \cap V \neq \emptyset. \qedhere
\]
\end{proof}

\begin{lemma}
If $V$ is an open of $A \times B$ and $\phi : A \tto B$ a lower hemicontinuous set-valued function, then,
\[
    \psi: a \in A \mapsto \{ b \in B, ~(a,b) \in V, ~b \in \phi(a) \} = \mathrm{s}V \cap \phi,
\]
is lower hemicontinuous.
\end{lemma}

\begin{proof}
Let $a_0 \in A$ and $W$ be an open set of $B$ such that $\psi(a_0) \cap W \neq \emptyset$. Let then $b_0 \in \psi(a_0) \cap W$. Let $\epsilon > 0$ be such that $\mathcal{B}((a_0,b_0), \epsilon) \subset V$ and $\mathcal{B}(b_0, \epsilon) \subset W$. By lower hemicontinuity of $\phi$, there exists $I \subset A$ open neighbourhood of $a_0$ such that for all $a \in I$,
\[
    \phi(a) \cap \mathcal{B}(b_0, \nicefrac{\epsilon}{2}) \neq \emptyset,
\]
namely for all $a \in I$, there exists $b(a) \in \phi(a) \cap \mathcal{B}(b_0, \nicefrac{\epsilon}{2})$. Without loss of generality we assume that $I \subset \mathcal{B}(a_0, \nicefrac{\epsilon}{2})$, then as,
\[
    \mathcal{B}(a_0, \nicefrac{\epsilon}{2}) \times \mathcal{B}(b_0, \nicefrac{\epsilon}{2})
    \subset \mathcal{B}((a_0,b_0), \epsilon)
    \subset V,
\]
for all $a \in I$,
\[
    b(a) \in \psi(a) \cap W \neq \emptyset. \qedhere
\]
\end{proof}

With these three lemmas, we conclude that $K : t \in I \mapsto \overline{\mathrm{s}V(t) \cap \mathcal{M}(t)}$ is lower hemicontinuous.

% $K : t \mapsto h^{-1}_t(\{0\}) \cap \bar{V}(t)$ is lower hemicontinuous. Indeed, $t \mapsto h^{-1}_t(\{0\})$ is lower hemicontinuous, therefore $t \mapsto h^{-1}_t(\{0\}) \cap \bar{V}(t)$ is lower hemicontinuous, then so is,
% \[
% K = \mathrm{s}\bar{V} \cap \mathcal{M} = \overline{\mathrm{s}V \cap \mathcal{M}}.
% \]
% Note that we can write that as for all $t \in I$,
% \[
% \overline{\mathrm{s}V(t) \cap \mathcal{M}(t)} \subset \mathrm{s}\bar{V}(t) \cap \mathcal{M}(t),
% \]
% and if,
% \[
% x \in (\mathrm{s}\bar{V}(t) \cap \mathcal{M}(t))
% \setminus \overline{\mathrm{s}V(t) \cap \mathcal{M}(t)},
% \]
% then

% BEWARE $\overline{A \cap B} \neq \bar{A} \cap \bar{B}$, REWRITE USING $\overline{\mathrm{s}V(t) \cap \mathcal{M}(t)}$ OR PROVE DIFFERENTLY

\subsection{Upper hemicontinuity of $K$}

% \begin{lemma}
% A locally bounded set-valued function $\phi : A \tto B$ with closed graph is upper hemicontinuous.
% \end{lemma}

\begin{lemma}
A locally bounded set-valued function $\phi : A \tto B$ is close-valued and upper hemicontinuous if and only if its graph is closed.
\end{lemma}

\begin{proof}
First we show that the closed graph property implies upper hemicontinuity and continuous values, assume then the graph of $\phi$ is closed. Let $a \in A$ and $V$ be an open set containing $\phi(a)$. We also denote $I$ a neighbourhood of $a$ and $D$ a compact containing $\phi(I)$, as $\phi$ is locally bounded. Suppose there is no neighbourhood $J$ of $a$ such that $\phi(J) \subset V$, then there exists a sequence $(a_n)$ such that $a_n \to_n a$ and $\phi(a_n) \cap (D \setminus V) \neq \emptyset$, we let then $(b_n)$ be a sequence of elements in these intersections. By compacity of $D$ in which $(b_n)$ belongs, we can assume without loss of generality that it converges to some limit, $b_n \to_n b \in D$. But then we have,
\[
    b_n \in \phi(a_n), ~a_n \to_n a, ~b_n \to_n b,
\]
by closure of the graph of $\phi$, this means $b \in \phi(a)$, but by closure of $D \setminus V$, it must be that $b \in D \setminus V$, all the while $b \in \phi(a) \subset V$, which is absurd. Finally, for all $a \in A$, $\phi(a) = \Gamma(\phi) \cap (\{a\} \times B)$ is closed as intersection of closed sets, with $\Gamma(\phi)$ denoting the graph of $\phi$.

\medskip

Assume now $\phi$ is upper hemicontinuous with closed values. Let $(a_n,b_n)_{n \in \mathbb{N}}$ be a sequence such that $(a_n, b_n) \to_n (a, b)$ some limit, while $b_n \in \phi(a_n)$. Let $W$ be an open containing $\phi(a)$, by upper hemicontinuity, there exits an open interval $I$ containing $a$ such that for all $a' \in I$, $\phi(a') \subset W$. As $a_n \to_n a$, for $n$ large enough $a_n \in I$, so $b_n \in \phi(a_n) \subset W$, therefore $b_n \in \bar W$. In turn, by taking $\epsilon > 0$ to $0$ using,
\begin{align*}
W_\epsilon &= \{ y \in B, ~d(y, \phi(a)) < \epsilon \} \\
&= \bigcup_{y \in B} B(y, \epsilon),
\end{align*}
we find that $b \in \overline{\phi(a)} = \phi(a)$, since $\phi$ has closed values.
\end{proof}

\begin{lemma}
If $\overline{\mathrm sV}$ is upper hemicontinuous, then so is $\overline{\mathrm sV} \cap \mathcal{M}$.
% If $\phi, \psi : A \tto B$ have closed values, one of them is locally bounded and one of them is upper hemicontinuous, 
% %and $\psi : A \tto B$ has closed values,
% then $\phi \cap \psi$ is upper hemicontinuous.
\end{lemma}
% If $\overline{\mathrm{s}V}$ is upper hemicontinuous, then so is $\overline{\mathrm{s}V \cap \mathcal{M}}$, for any correspondence $\mathcal{M}$.
% \end{lemma}

\begin{proof}
By the previous lemma, $\overline{\mathrm sV}$ being locally bounded with closed values and upper hemicontinuous, its graph is closed. As a result, the graph of $\overline{\mathrm sV} \cap \mathcal{M}$,
\[
\Gamma(\overline{\mathrm sV} \cap \mathcal{M})
= \Gamma(\overline{\mathrm sV}) \cap h^{-1}(\{0\}),
\]
is closed. We add that $\overline{\mathrm sV} \cap \mathcal{M}$ is locally bounded, then by the previous lemma $\overline{\mathrm sV} \cap \mathcal{M}$ is upper hemicontinuous.
% TO BE PROVEN, or directly assume the graph of $\overline{\mathrm{s}V \cap \mathcal{M}}$ is closed.
\end{proof}

With the assumption that at all times $t$,
% \[
% \partial (\mathrm sV(t)) \cap \mathcal{M}(t) = \emptyset,
% \]
\[
\partial (\overline{\mathrm sV(t)} \cap \mathcal{M}(t)) = \emptyset,
\]
we get that,
\[
\partial (\mathrm sV(t) \cap \mathcal{M}(t)) = \emptyset,
\]
thus,
\[
\overline{\mathrm sV \cap \mathcal M}
\]
\[
\mathrm sV(t) \cap \mathcal{M}(t)
= \overline{\mathrm sV(t)} \cap \mathcal{M}(t),
\]
therefore,
\[
\overline{\mathrm sV \cap \mathcal{M}}
= \overline{\overline{\mathrm sV} \cap \mathcal{M}}
= \overline{\mathrm sV} \cap \mathcal{M}.
\]
With the additional assumption that $\overline{\mathrm sV}$ is upper hemicontinuous and the previous lemma,
\[
K = \overline{\mathrm sV \cap \mathcal{M}},
\]
is upper hemicontinuous.
%With this last lemma and the assumptions that $\mathrm sV$ is upper hemicontinuous and $\overline{\mathrm sV} \cap \mathcal{M} = \overline{\mathrm sV} \cap \mathcal{M} we see that $K = \overline{\mathrm{s}V \cap \mathcal{M}}$ is upper hemicontinuous with the assumption that $\overline{\mathrm{s} V}$ is upper hemicontinuous.

% \begin{lemma}
% Let $\phi : A \tto B$ be a locally bounded set-valued function with closed graph and $\psi : A \tto B$ a set-valued function with closed graph. Then, $\phi \cap \psi$ is locally bounded with closed graph, thus upper hemicontinuous.
% \end{lemma}

% \begin{proof}
% It is clear that the graph of $\phi \cap \psi$ is the intersection of the graph of $\phi$ and $\psi$, thus is closed. Now for $a \in A$, let $I$ is a neighbourhood of $a$ and $D$ be a compact of $B$ such that $\phi(I) \subset D$. Then surely,
% \[
% (\phi \cap \psi)(I) \subset \phi(I) \subset D,
% \]
% so that $\phi \cap \psi$ is also locally bounded.
% \end{proof}

%With these two lemmas, we see first that $\bar{V}$ being locally bounded and of closed graph, it is upper hemicontinuous, then
%$\mathrm{s}\bar{V} \cap \mathcal{M}$ is upper hemicontinuous and locally bounded.
%$K = \bar{V} \cap \mathcal{M}$ is also upper hemicontinuous and locally bounded.

% \begin{lemma}
% If $\phi : A \tto B$ is upper hemicontinuous, then so is $\bar{\phi}$.
% \end{lemma}

% \begin{proof}
% Let $a_0 \in A$ and $V$ open containing $\bar{\phi}(a_0)$
% \end{proof}

\subsection{Construction of $\hat f$ and $\hat K$}

In the following proof we construct $\hat K$ which acts as a prison (a jail through time), simply from the definition of a basin.

\begin{proof}[Proof of Proposition \ref{prop:basin}]
Let $t \in I$ and define,
\[
    \hat{f}(t) = \frac{1}{2} \left(
    f^*(t) +
    \min_{\partial (\mathrm sV(t)) \cap \mathcal M(t)} f_t
    \right) > f^*(t),
\]
as otherwise we would have,
\[
f^*(t) = \min_{K(t)} f_t = \min_{\partial (\mathrm sV(t)) \cap \mathcal M(t)} f_t,
\]
thus,
\[
    \partial (\mathrm sV(t)) \cap K^*(t) \neq \emptyset,
\]
contradicting the hypothesis that,
\[
    K^*(t) \subset \mathrm sV(t),
\]
as $\mathrm sV(t)$ is open.

Assume that there is no $\epsilon > 0$ such that for all $t \in I$,
\[
    \hat f(t) - f^*(t) \geq \epsilon.
\]
Then there exists a sequence $(t_n)$ of $I$ compact such that $\hat f(t_n) - f^*(t_n) \to_n 0$, without loss of generality $t_n \to_n \bar t \in I$ some limit. We have $t_n \to_n \bar t$ and,
\[
\min_{\partial (\mathrm sV(t_n)) \cap \mathcal M(t_n)} f_{t_n} \to_n f^*(\bar t).
\]
Call $x_n$ a minimum at time $t_n$. As $\mathrm K$ is locally bounded, without loss of generality, $x_n \to_n \bar x \in K(\bar t)$ some limit as the graph of $K$ is closed. By continuity,
\[
f(x_n, t_n) \to_n f(\bar x, \bar t) = f^*(\bar t),
\]
thus $\bar x \in K^*(\bar t)$. We can show however that $\bar x \in \partial (\mathrm sV(\bar t))$, which is a contradiction with the previous fact, hence the existence of $\epsilon > 0$.

Indeed, since $\overline{\mathrm sV}$ has a closed graph (since locally bounded, upper hemicontinuous with closed values), then so is the graph of $\partial (\mathrm sV)$, as it can be written,
\[
\Gamma(\partial (\mathrm sV)) = \Gamma(\overline{\mathrm sV} \setminus \mathrm sV) = \Gamma(\overline{\mathrm sV}) \setminus V.
\]

\medskip

Define now, for all $t \in I$,
\[
    \hat K(t) = \{ x \in K(t), ~f(x,t) \leq \hat f(t) \} \subset K(t).
\]
We simply now need to show that $\hat K(t) \subset \mathrm sV(t)$. Well, clearly for all $x \in \hat K(t)$, $x \in K(t) = \overline{\mathrm sV(t)} \cap \mathcal M(t)$ and,
\[
    f(x,t) \leq \hat f(t) < \min_{\partial (\mathrm sV(t)) \cap \mathcal M(t)} f_t,
\]
so that $x \not \in \partial (\mathrm sV(t))$, thus $x \in \mathrm sV(t)$.
\end{proof}

\clearpage

\section{Proof of Theorem \ref{thm:ode}}
\label{app:ode}

\begin{lemma}
\label{lem:proj}
$P = I_n - J^\top (J J^\top)^{-1} J$ is the orthogonal projection on $\ker J$, thus,
\begin{equation}
    I_n \succeq P \succeq 0.
\end{equation}
\end{lemma}

\begin{proof}[Proof of Lemma \ref{lem:proj}]
Let $x \in \ker J$, then,
\[
    (I_n - J^\top (J J^\top)^{-1} J)x = x.
\]
Now since $(\ker J)^\perp = \Ima J^\top$, we let $x = J^\top y$ be an element of $(\ker J)^\perp$, then,
\begin{align*}
    (I_n - J^\top (J J^\top)^{-1} J)x
    &= (I_n - J^\top (J J^\top)^{-1} J) J^\top y \\
    &= 0. \qedhere
\end{align*}
\end{proof}

\begin{proposition}[Continuous dependency on initial condition]
\label{prop:continuityflow}
Let $g : U \times \mathbb R_+ \subset \mathbb{R}^n \times \mathbb{R} \to \mathbb{R}^n$ be continuously differentiable on $U \times \mathbb R_+$ open %(with $0 \in V$)
and consider the ODE,
\begin{equation}
    \label{eq:continuityflowode}
    \dot{x} = g(x,t).
\end{equation}
Let $x_0 \in U$ and $(I,x)$ be a maximal solution to \eqref{eq:continuityflowode} with initial condition $x(0) = x_0$. Then, for all $[0,T] \subset I$ closed interval and all $\epsilon > 0$, there exists $\delta > 0$ such that the flow $\varphi$ is defined on $[0,T] \times B(x_0, \delta)$, and for all $y_0 \in B(x_0, \delta)$ and $t \in [0,T]$,
\[
    \| \varphi^t(y_0) - \varphi^t(x_0) \| \leq \epsilon.
\]
\end{proposition}

This proposition allows us to consider \eqref{eq:odescaled} as an approximation of \eqref{eq:odescaledflat}. Indeed, we can instead consider the following ODE, which passes $\alpha$ as an initial condition,
\begin{align*}
     \dot{x}(s) =&~ - P \nabla_x f(x(s), \alpha(s) s+t_0) \\
     &~ - \alpha(s) J^\top (J J^\top)^{-1} h'(\alpha(s) s+t_0) \\
     \dot{\alpha}(s) =&~ 0.
     \tag{ODE$_\alpha^*$}\label{eq:odescaledparam}
\end{align*}

\begin{proof}[Proof of Proposition \ref{prop:continuityflow}]
Let $\delta > 0$ and $T >0$ be such that $[0,T] \subset I$. Without loss of generality, we assume that for all $t \in [0,T]$,
\[
\bar{B}(\varphi^t(x_0),\epsilon) \subset U.
\]
Surely, as long as it is defined,
\[
    \varphi^t(u) = u + \int_0^t g(\varphi^s(u),s) \mathrm{d}s,
\]
therefore,
\begin{align}
    &~ \| \varphi^t(u) - \varphi^t(v) \| \\
    \leq&~ \|u-v\| + \int_0^t \| g(\varphi^s(u),s) - g(\varphi^s(v),s) \| \mathrm{d}s \\
    \leq&~ \|u-v\| + \Lambda_\delta \int_0^t \| \varphi^s(u) - \varphi^s(v) \| \mathrm{d}s,
\end{align}
where,
\[
    \Lambda = \max_{t \in [0,T], \|z-\varphi^t(x_0)\| \leq \epsilon} \vertiii{\mathcal{J}_g(z,t)},
\]
and for any linear map $A$, the operator norm is defined as,
\[
    \vertiii{A} = \max_{x, ~\|x\|=1} \|Ax\|.
\]
Therefore, by Gr\"onwall's lemma, for $t \in [0,T]$,
\begin{equation} \label{eq:gronresult}
    \| \varphi^t(u) - \varphi^t(v) \|
    \leq \|u-v\| e^{\Lambda t}
    \leq \|u-v\| e^{\Lambda T},
\end{equation}
again, as long as the flow is defined. 

Decrease now $\delta > 0$ if necessary, such that,
\[
    %\bigcup_{t \in [0,T]} \bar{B}(\varphi^t(x_0), \delta e^{\Lambda T}) \subset U \textrm{ and } 
    \delta e^{\Lambda T} \leq \epsilon.
\]
Let $y_0 \in B(x_0, \delta)$ and $(J,y)$ be a maximal solution to \eqref{eq:continuityflowode} with initial condition $y(0) = y_0$. Applying \eqref{eq:gronresult} and the escape dilemma \eqref{lem:escape} to $y$ with the compact,
\[
    \bigcup_{t \in [0,T]} \bar{B}(\varphi^t(x_0), \epsilon),
\]
we deduce that $[0,T] \subset J$. As a result, the flow is defined on $[0,T] \times B(x_0, \delta)$ and for all $t \in [0,T]$ and $u,v \in B(x_0, \delta)$,
\[
    \| \varphi^t(u) - \varphi^t(v) \|
    \leq \|u-v\| e^{\Lambda T} \leq \epsilon. \qedhere
\]
\end{proof}

% \begin{lemma}
% \label{lem:compdis}
% If $(K_i)_{i \in \nn{1,n}}$ is a finite family of disjoint compacts of a metric space, there exists $\epsilon > 0$ such that the sets $(K^\epsilon_i)_{i \in \nn{1,n}}$ are disjoint.
% \end{lemma}

% \begin{proof}
% Indeed, the map $d:(x,y) \in K_1 \times K_2 \mapsto d(x_1, x_2)$ is continuous on a compact thus admits a minimum. However since $K_1 \cap K_2 = \emptyset$, this minimum distance cannot be null. Repeating this process, we find a positive lower bound for each pair of compacts, taking the minimum of them and dividing it by three, we find a suitable $\epsilon$.
% \end{proof}

\begin{proof}[Proof of Theorem \ref{thm:ode}]
For the sake of readability, we let $g_\alpha$ denote the right-hand side of \eqref{eq:ode}, while $g_\alpha^t$ (respectively $g_0^t$) will denote the right-hand side of \eqref{eq:odescaled} (respectively \eqref{eq:odescaledflat}) at time $t$ (instead of $t_0$).

\proofpart{Uniform continuity of $f^*$}
As $K$ is continuous, of compact non-empty values, $K^*$ is upper hemicontinuous of compact non-empty values and $f^*$ is continuous. Since $f^*$ is continuous on the compact $I$, it is uniformly continuous, thus without loss of generality $\nu > 0$ is such that for all $t,t' \in I$,
\begin{equation}
    |t-t'| \leq \nu \implies | f^*(t) - f^*(t')| \leq \epsilon.
    \label{eq:unifcontf}
\end{equation}
This justifies that if the solution of \eqref{eq:odescaled} somehow tracks with a delay, at least in the proof, it still tracks $f^*$ at its current value.

\proofpart{Uniform convergence of $\psi_0$}

% Consider now $\hat{f} - f^*$, it is continuous with positive values. It must attain a minimum on the compact $J$, thus without loss of generality we assume henceforth that, for all $t \in J$,
% \[
%     \hat{f}(t) - f^*(t) \geq 4\epsilon.
% \]
We are now extending the Lyapunov argument uniformly in time. In other words, we will find $T$ (called then later $\bar s$) much like proposition \ref{prop:lyap}, but valid at all time $t \in I$.

Reducing $\epsilon > 0$ if necessary, without loss of generality we assume henceforth that, for all $t \in I$,
\[
    \hat{f}(t) - f^*(t) \geq 4\epsilon.
\]
% Further, it also attains a maximum on $J$ whose value will be denoted $R$, in reference to proposition \ref{prop:lyap}. Define also,
% \[
%     M^\epsilon : t \in I \mapsto \{ y \in K(t), ~f(y,t) < f^*(t) + \epsilon \}.
% \]
As $K$ is continuous with nonempty compact values and $f$ 
%$(x,t) \in \mathbb{R}^n \times I \mapsto f(x,t)$
is continuous,
\[
    t \in I \mapsto \max_{K(t)} f_t - f^*(t),
\]
is continuous as well, which allows us to define,
\[
    R = \max_{t \in I, ~K(t)} f_t - f^*(t).
\]
We then have, for all $t \in I$,
\[
    R \geq \hat{f}(t) - f^*(t) \geq 4\epsilon.
\]
%$R$, in reference to proposition \ref{prop:lyap}. 
Define also,
\[
    K^\epsilon : t \in I \mapsto \{ y \in K(t), ~f(y,t) < f^*(t) + \epsilon \}.
\]
We verify that $K \setminus K^\epsilon$ is an upper hemicontinuous function. We know it is locally bounded so it suffices to show its graph is closed. Let then $t_n \to_n t$ and $x_n \to_n x$ be sequences converging to some limits such that $x_n \in K(t_n) \setminus K^\epsilon(t_n)$, then,
\[
    x_n \in K(t_n) \text{ and } f(x_n, t_n) \geq f^*(t_n) + \epsilon,
\]
therefore, by continuity of $K$ and $f,f^*$,
\[
    x \in K(t) \text{ and } f(x, t) \geq f^*(t) + \epsilon,
\]
thus $x \in K(t) \setminus K^\epsilon(t)$, therefore $K \setminus K^\epsilon$ is indeed upper hemicontinuous. Its values are also non-empty and compact. As a result, by the maximum theorem, the following function is lower semicontinuous,
\[
    t \in I \mapsto - \max_{K(t) \setminus K^\epsilon(t)} \nabla_x f_t^\top g_0^t
\]
and positive. Thus, it attains a minimum on the compact $J$, of value $\lambda > 0$, again so called after proposition \ref{prop:lyap}. Define then $\bar{s} = \nicefrac{R}{\lambda}$, following the same argument as proposition \ref{prop:lyap},
\begin{equation}
    \forall t \in I, x \in K(t), ~f(\psi^{\bar{s},t}_0(x), t) - f^*(t) \leq \epsilon.
    \label{eq:scaled}
\end{equation}

\proofpart{Continuous dependence on $\alpha$}
We want to establish the continuity of the solution of \eqref{eq:odescaled} with respect to the parameter $\alpha$. Let then $t \in I$, $x \in K(t)$, $\tau \in [0,\bar{s}]$, $\alpha \in [0,1]$, we have, as long as it is defined,
\[
    \psi^{\tau, t}_\alpha(x)
    = x + \int_0^\tau g^t_\alpha(\psi^{\sigma, t}_\alpha(x), \sigma) \mathrm{d} \sigma.
\]
Therefore, as long as $\psi^{s, t}_\alpha(x)$ is defined 
and $\| \psi^{s, t}_\alpha(x) - \psi^{s, t}_0(x) \| \leq 1$ 
%and $d(\psi^{s, t}_\alpha(x), K(t + \alpha s)) \leq 1$ 
for all $s \in [0,\tau]$,
\begin{align*}
    &~\| \psi^{\tau, t}_\alpha(x) - \psi^{\tau, t}_0(x) \| \\
    \leq&~
    \int_0^\tau \| g^t_\alpha(\psi^{\sigma, t}_\alpha(x), \sigma) - g^t_\alpha(\psi^{\sigma, t}_0(x), \sigma) \| \mathrm{d} \sigma \\
    &~ + \int_0^\tau \| g^t_\alpha(\psi^{\sigma, t}_0(x), \sigma) - g^t_0(\psi^{\sigma, t}_0(x), \sigma) \| \mathrm{d} \sigma \\
    \leq &~
    \Lambda \int_0^\tau \| \psi^{\sigma, t}_\alpha(x) - \psi^{\sigma, t}_0(x) \| \mathrm{d} \sigma + A \alpha,
\end{align*}
where,
\[
    A = \bar{s} \max_{t \in I} \max_{x \in K(t)} \max_{\sigma \in [0,\bar{s}]} \max_{\alpha \in [0,1]} \left\| \frac{\partial g_\alpha^t}{\partial \alpha}(x, \sigma) \right\|,
\]
and $\Lambda$ is defined by,
\[
    \max_{t \in I} \max_{\sigma \in [0,\bar{s}]}  \max_{x \in K(t)} \max_{\alpha \in [0,1]} \max_{x\in\mathbb{R}^n, d(x,K(t)) \leq 1} \vertiii{\mathcal{J}_{g_\alpha^t} (x, \sigma)}.
\]
% As a reminder, for all $\alpha \geq 0$, $x \in \mathbb{R}^n$ and $t \geq 0$,
% \[
%     g^\alpha(x, t) = % \frac{1}{\alpha} g^t_\alpha(s).
%      - P \nabla_x f(x, t)
%     - \alpha J^\top (J J^\top)^{-1} h'(x,t) ,
% \]
% where $J,P$ are evaluated at $(x,t)$, and for all $s \geq 0$ and $\alpha > 0$,
% \[
%      g^t_\alpha(s) = \alpha g^\alpha(x, \alpha s + t).
%     %[ - P \nabla_x f(x, t)
%     %- \alpha J^\top (J J^\top)^{-1} h'(x,t) ].
% \]
As a reminder, for all $x \in \mathbb{R}^n$ and $\alpha,s,t \geq 0$,
\[
    g_\alpha^t(x,s)
    = - P \nabla_x f(x, \alpha s+t) - \alpha J^\top (J J^\top)^{-1} h'(x,\alpha s+t),
\]
where $J,P$ are evaluated at $(x, \alpha s +t)$. $g_\alpha^t(x,s)$ is thus continuously differentiable in $x$, $\alpha$, $s$ and $t$, justifying the given bounds.
% , and,
% \[
%     g^\alpha(x, \alpha s + t) = \frac{1}{\alpha} g^t_\alpha(s).
%     %[ - P \nabla_x f(x, t)
%     %- \alpha J^\top (J J^\top)^{-1} h'(x,t) ].
% \]
As a result, by Gr\"onwall's lemma,
\[
    \| \psi^{\tau, t}_\alpha(x) - \psi^{\tau, t}_0(x) \|
    \leq A \alpha e^{\Lambda \tau},
\]
again, as long as $\psi^{s, t}_\alpha(x)$ is defined 
and $\| \psi^{s, t}_\alpha(x) - \psi^{s, t}_0(x) \| \leq 1$ 
%and $d(\psi^{s, t}_\alpha(x), K(t + \alpha s)) \leq 1$ 
for all $s \in [0,\tau]$. Define $\bar{\alpha} = \min(1, \nicefrac{e^{-\Lambda \bar{s}}}{A})$, then with the addition of the escape dilemma (lemma \ref{lem:escape}), this ensures $\psi^{\tau, t}_\alpha(x)$ is well-defined for all $t \in I$, $x \in K(t)$, $\tau \in [0, \bar{s}]$ and $\alpha \in [0,\bar{\alpha}]$ and,
\[
    \| \psi^{\tau, t}_\alpha(x) - \psi^{\tau, t}_0(x) \|
    \leq A \bar{\alpha} e^{\Lambda \bar{s}}.
\]
Without loss of generality, we can reduce $\bar{\alpha}$ such that,
\[
    A \bar{\alpha} e^{\Lambda \bar{s}} \leq \nicefrac{\epsilon}{L},
\]
then,
% , as long as $\alpha \tau + t \leq t_1$. 
% If we define $\bar{\alpha} = \min(1, \nicefrac{e^{-\Lambda \bar{s}}}{A})$, we then have,
\begin{align}
    \forall \alpha \leq \bar{\alpha}, 
    ~t \in I, ~x \in K(t), ~s \in [0,\bar{s}], \notag\\
    ~ \| \psi^{s,t}_\alpha(x) -  \psi^{s,t}_0(x) \|
    \leq \nicefrac{\epsilon}{L}. \label{eq:cont}
\end{align}

\proofpart{The prison argument}
If $\phi$ is a continuous function defined on an interval $J = [t_0,t_2] \subset I$ satisfying $h(\phi(t),t)=0$, $f(\phi(t),t) - f^*(t) \leq 3 \epsilon$ for all $t \in J$, and $\phi(t_0) \in K(t_0)$, then,
\begin{equation}
    \forall t \in J, ~\phi(t) \in \hat{K}(t).
\end{equation}
In other words, $\phi$ cannot escape $\hat{K}$ once it has entered and as long as it satisfies the constraint ($h = 0$) and its $f$-value is strictly less than the edge value $\hat{f}$.

Indeed, otherwise denote,
\[
    t_3 = \inf_{t \in J, ~\phi(t) \not\in \hat{K}(t)} t.
\]
By upper hemicontinuity of $\hat{K}$, it must be that $\phi(t_3) \in \hat{K}(t_3)$, as either $t_3 = t_0$, or there exists a sequence $(t_n)$ such that $t_n \to_n t_3$ and $\phi(t_n) \in \hat{K}(t_n)$, which implies $\phi(t_3) \in \hat{K}(t_3)$. Let $V$ isolate $\hat{K}(t_3)$ from the lower level set of value $\hat{f}(\phi(t_3))$. By continuity, there exits $\eta > 0$ such that $\phi([t_3, t_3 + \eta]) \subset V$, but then for $t \in [t_3, t_3 + \eta]$,
\[
    h(\phi(t), t) = 0, ~f(\phi(t),t) \leq f^*(t) + 3\epsilon \leq \hat{f}(t), ~(t,\phi(t)) \in V,
\]
therefore,
\[
    \phi(t) \in \hat{K}(t),
\]
which contradicts the definition of $t_3$, thus is absurd.

We want to verify that $\varphi_\alpha$ is defined over $J \times K(t_0)$ for $\alpha \in (0,\bar{\alpha}]$, and further that for all $t \in [t_0 + \nu, t_1]$ and $x \in K(t_0)$, 
\[
    \varphi^t_\alpha(x) \in K(t).
\]
For the second point, once it has been shown to indeed be defined, we simply need to show that for all $t \in J$, and $x \in K(t_0)$,
\[
    h(\varphi^t_\alpha(x),t) = 0
    \text{ and }
    f(\varphi^t_\alpha(x),t) \leq f^*(t) + 3 \epsilon,
\]
as we just saw. The first equality comes easily from,
\begin{align*}
    \frac{\mathrm{d}}{\mathrm{d}t} h(\varphi_\alpha^t(x),t)
    = &~ - \frac{1}{\alpha} J(I_n - P) \nabla_x f(\varphi_\alpha^t(x), t) \\
    &~ - J J^\top (J J^\top)^{-1} h'(\varphi_\alpha^t(x),t) \\
    &~ + h'(\varphi_\alpha^t(x),t) \\
    =&~ 0,
\end{align*}
while $h(x,t_0) = 0$. The other fact is proven in the next part.

\proofpart{Patching parts}
Now, let us gather all results.
Let then $\alpha \in (0, \bar{\alpha}]$, $x \in K(t_0)$ and $s \in [0, \min(\bar{s}, \nicefrac{t_1-t}{\alpha})]$, result \eqref{eq:cont} guarantees the existence of $\psi^{s,t_0}_\alpha(x)$ and we have,
\[
    \| \psi^{s,t_0}_\alpha(x) - \psi^{s,t_0}_0(x)) \| \leq \nicefrac{\epsilon}{L}.
\]
Since $f$ is $L$-Lipschitz, $f^*$ is uniformly continuous and \eqref{eq:odescaled} resembles \eqref{eq:odescaledflat} for $\alpha$ small enough,
\begin{align*}
    &~ f(\psi^{s,t_0}_\alpha(x), \alpha s + t_0) - f^*(\alpha s +t_0) \\
    =&~ f(\psi^{s,t_0}_\alpha(x), t_0) - f(\psi^{s,t_0}_0(x), t_0)
    + f(\psi^{s,t_0}_0(x), t_0) - f^*(t_0) \\
    &~+ f^*(t_0) - f^*(t_0 + \alpha s) \\
    \leq&~ 3\epsilon, 
\end{align*}
thus for all $s \leq \bar{s}$ such that $t_0 + \alpha s \leq t_1$,
\[
    f(\varphi^{t_0+\alpha s}_\alpha(x), t_0 + \alpha s) \leq f^*(t_0 + \alpha s) + 3 \epsilon.
\]
We deduce that $\varphi^{t_0+\alpha \bar{s}}_\alpha(x) \in K(t_0+\alpha \bar{s})$. Applying the exact same method again with time $t_0 + \alpha \bar{s}$ instead of $t_0$, we find for all $s \leq \bar{s}$ such that $t_0 + \alpha s \leq t_1$,
\[
    f(\varphi^{t_0+\alpha (\bar{s}+s)}_\alpha(x), t_0 + \alpha (\bar{s}+s)) \leq f^*(t_0 + \alpha (\bar{s}+s)) + 3 \epsilon,
\]
we can repeat until $J$ is covered so that for all $t \in J$,
\[
    f(\varphi^t_\alpha(x), t) \leq f^*(t) + 3 \epsilon. \qedhere
\]

\clearpage

\end{proof}

\section{Non-uniform convergence}

\begin{definition}
Given a flow $\varphi$ on $U \subset \mathbb{R}^n$ open, we say that $x$ is a wandering point if there exists an open neighbourhood $V$ of $x$ and a time $\bar{t} >0$ such that,
\[
    \forall t \geq \bar{t}, ~\varphi^t(V) \cap V = \emptyset
\]
Note that points $y \in V$ such that $\varphi^t(y)$ is not defined are omitted from the set $\varphi^t(V)$, to account for possible escapes.

We call $W$ the wandering set, the set of wandering points, and $M = U \setminus W$ the non-wandering set.
\end{definition}

% \begin{definition}
% A wandering set $V$ is a set such that there exists $\bar{t} > 0$ such that,
% \[
%     \forall t \geq \bar{t}, ~\varphi^t(V) \cap V = \emptyset.
% \]
% \end{definition}

\begin{lemma}[Convergence to $M$ for $X$ compact]
Let $X$ be a compact metric space and $\varphi$ be a flow defined for all positive times such that $\varphi^t(X) \subset X$ for all $t \geq 0$. Then, for all $\epsilon>0$ and for all $x \in X$, there exists $\bar{t} > 0$ such that for all $t \geq \bar{t}$
\[
    d(\varphi^t(x),M) \leq \epsilon.
\]
\end{lemma}

\begin{proof}
Let $\epsilon > 0$ and define the following compact set,
\[
    K = \{ x \in X, ~d(x,M) \geq \epsilon\} \subset W.
\]
For all $x \in K$, there exists an open neighbourhood $V_x$ of $x$ and $t_x > 0$ such that for all $t \geq t_x$,
\[
    \varphi^t(V_x) \cap V_x = \emptyset.
\]
By compactness of $K$, we can extract a finite cover of $K$ from $(V_x)_{x \in K}$, say given by $E = \{x_1, \dots, x_n \}$. We note $T = \max_{x \in E} t_x$. If $x \in K$, then $x$ belongs to some $V_y$ with $y \in E$. However $x$ can only stay for a period of time $T$ and never come back. It might reach another $V_{y'}$ with $y' \in E$ different than $y$, and again can only stay for a period of time $T$ and never visit again $V_{y'}$. Eventually, all $V_z$ with $z \in E$ are exhausted, so that $x$ reaches and remains in $X \setminus K$. 

If $x \in X \setminus K$, then either it remains in $X \setminus K$, or it leaves and the previous reasoning brings $x$ in $X \setminus K$ forever. As a result, any initial condition tends toward $M$.
\end{proof}

% we find an open neighbourhood $V$ of $x$ contained in $V_y$, then for all $t \geq T$,
% \[
%     \varphi^t(V) \cap V = \emptyset.
% \]

% and deduce that for $x \in W$, either $x \not \in C$ and then either $\varphi^t(x)$ remains in $K \setminus C$, or there exists $\bar{t} > 0$ such that $t>0$ such that $x(t) \in C$.

% First we establish that if $A,B$ are wandering, then so is $A \cup B$. Assume $A \cup B$ is not wandering, then there exists a sequence of positive times $(t_n)_{n \in \mathbb{N}}$ such that,
% \[
%     t_n \to_n \infty \text{ and } \forall n \in \mathbb{N}, ~\varphi^{t_n}(A\cup B) \cap (A \cup B) \neq \emptyset. 
% \]
% Without loss of generality, assume that,
% \[
%     \varphi^{t_n}(A\cup B) \cap B \neq \emptyset,
% \]
% and thus, since $B$ is wandering and,
% \[
%     \varphi^{t_n}(A\cup B) = \varphi^{t_n}(A) \cup \varphi^{t_n}(B),
% \]
% we have,
% \[
%     \varphi^{t_n}(A) \cap B \neq \emptyset.
% \]
% Denote then by $a_n$ a sequence of elements of $A$ such that $\varphi^{t_n}(a_n) \in B$. 

\begin{lemma}[Lyapunov wanderers]
\label{lem:wanderers}
Let $f : U \subset \mathbb{R}^n \to \mathbb{R}^n$ be locally Lipschitz on $U$ open. Consider the following ODE,
\begin{equation}
    \label{eq:generalode}
    \dot{x} = f(x),
\end{equation}
and its flow $\varphi$. If $g : U \to \mathbb{R}$ is continuously differentiable such that for all $x \in U$,
\[
    \nabla g(x)^\top f(x) \leq 0,
\]
then we have,
\[
    f^{-1}(\{0\}) \subset M \subset (\nabla g^\top f)^{-1}(\{0\}),
\]
where $^{-1}$ denotes the preimage.%all points $x$ such that $\nabla g(x)^\top f(x) < 0$ are wandering.
\end{lemma}
The function $g$ in this lemma act in a similar fashion as a Lyapunov function except we do not use this argument around a strict local minimum of $g$, but rather show all non-`$\nabla g^\top f$-critical' points are wandering.

\begin{proof}
Let $x^* \in U$ be such that,
\[
    \nabla g(x^*)^\top f(x^*) < 0,
\]
thus also $\|f(x^*)\| > 0$. Let then,
\[
    \epsilon = \frac{1}{2} | \nabla g(x^*)^\top f(x^*) | > 0
    \text{ and } \upsilon = \frac{1}{2} \| f(x^*) \| > 0.
\]
By continuity of $\nabla g^\top f$ and $f$, there exists $d > 0$ such that $\|x - x^*\| \leq d$ implies,
\[
| \nabla g(x)^\top f(x) | > \epsilon \text{ and } 3 \upsilon > \|f(x)\| > \upsilon.
\]
Call $B = B(x^*, d)$, the open ball of radius $d$ centered in $x^*$ and $V = B(x^*, \frac{d}{2}) \subsetneq B$. Let $x_0 \in V$ and $(I,x)$ be the unique maximal solution of \eqref{eq:generalode} with initial condition $x(0) = x_0$. Let $t >0$ be such that $x([0,t]) \subset \bar{B}$, then
\[
\eta \triangleq 2 \max_{\bar{B}} g > g(x(0)) - g(x(t)) = - \int_0^t \nabla g(x(s))^\top f(x(s)) \mathrm{d}s \geq \epsilon t.
\]
In particular, invoking the escape dilemma (lemma \ref{lem:escape}), we conclude that there exists $\frac{\eta}{\epsilon} \geq t > 0$ such that $x(t) \not \in \bar{B}$. Otherwise, the solution would be defined for all positive times, eventually contradicting $\eta > \epsilon t$.
\bigskip

By continuity of $x$ (and thus of $\|x-x^*\|$), there exists $t_2 > t_1 > 0$ such that,
\[
    x(t_1) \in \partial V \text{ and } x(t_2) \in \partial B.
\]
The time spent to leave $B$ must then be at least,
\begin{align*}
t_2 - t_1  &\geq \frac{1}{3 \upsilon} \int_{t_1}^{t_2} \| f(x(s)) \| \mathrm{d}s\\
&\geq \frac{1}{3 \upsilon} \left\| \int_{t_1}^{t_2} - f(x(s)) \mathrm{d}s \right\| = \frac{1}{3 \upsilon} \|x(t_2) - x(t_1)\| \\
& \geq \frac{1}{3 \upsilon} \| x(t_2) - \bar{x} \| - \frac{1}{3 \upsilon} \| x(t_1) - \bar{x} \| \geq \frac{d}{6 \upsilon},
\end{align*}
so that $g \circ x$ will have decreased during that passage, at least by,
\[
g(x(t_1)) - g(x(t_2)) = - \int_{t_1}^{t_2} \nabla g(x(s))^\top f(x(s)) \mathrm{d}s \geq \epsilon (t_2 - t_1) \geq \frac{\epsilon d}{6 \upsilon} > 0.
\]
\bigskip

Now, by continuity of $g$, there exists $\frac{d}{2} > \delta > 0$ such that $\|x-x^*\| < \delta$ implies,
\[
    |g(x) - g(x^*)| \leq \frac{\epsilon d}{14 \upsilon}.
\]
Define then $W = B(x^*, \delta) \subsetneq V$ and assume $x_0 \in W$. Then, for all time $t \geq t_2$,
\begin{align*}
    |g(x(t)) - g(x^*)|
    &\geq - |g(x(0)) - g(x^*)| + g(x(0)) - g(x(t)) \\
    &\geq - \frac{\epsilon d}{14 \upsilon} + g(x(t_1)) - g(x(t_2)) \\
    &\geq - \frac{\epsilon d}{14 \upsilon} + \frac{\epsilon d}{6 \upsilon} > \frac{\epsilon d}{14 \upsilon},
\end{align*}
so that $x(t) \not \in W$. Thus $\varphi(W,t) \cap W = \emptyset$ for all $t \geq t_2$, thus $x^*$ is wandering.
\medskip

Finally, if $x^*$ is such that $f(x^*) = 0$, then $x^*$ is an equilibrium, thus a non-wandering point.
\end{proof}

\newpage

% A lead to find a time $\bar{s}$ such that for all $s \geq \bar{s}$ and locally in $t_0$, 
% \[
%     d(x_{t_0}(s), \mathcal{C}(f_{t_0}, h_{t_0})) \leq \epsilon,
% \]
% is to consider the initial condition $\bar{x}$ is $\epsilon$-close to the solution of before, thus still is in a compact $K$. Consider $\mathcal{C}(f,h)_\epsilon$ the $\epsilon$-enlarged set (compact). By continuity, $\nabla g^\top f$ reaches a minimum on it say $c <0$, then the level set,
% \[
%     L \triangleq (\nabla g^\top f)^{-1}((-\infty, \nicefrac{c}{2}]) \cap K
% \]
% is compact and $\mathcal{C}(f,h) \subset K \setminus L \subset \mathcal{C}(f,h)_\epsilon$. Then on $L$ the rate of decrease of $g(x)$ is at least $- \frac{c}{2} > 0$, this should give, added the initial conidtion is in $K$, a universal $\bar{s}$ after which the solution $x(s)$ stays $\epsilon$-close from $\mathcal{C}(f,h)$. Additionally, this should be doable locally in $t$ around $t_0$. 

In this section, the time $t_0$ is fixed so that $f$ refers to $f(\cdot, t_0)$ and likewise for $h$. Assume $x$ is a solution of \eqref{eq:odescaledflat}, for some initial condition. Informally, $f$ is a Lyapunov function in the sense that $s \mapsto f(x(s))$ is non-increasing, and actually is decreasing if $x(s)$ does not satisfy the necessary KKT conditions. If we add some assumptions on the landscape $f$ draws, we will prove that almost all initial conditions make $x$ converge to a local minima.

\begin{definition}
Let $f: U \subset \mathbb{R}^n \to \mathbb{R}$ and $h: U \subset \mathbb{R}^n \to \mathbb{R}^m$ be continuous, with $U$ open. We say that $x \in U$ is a constrained local minimum of $f$ (subject to $h$) if $h(x) = 0$ and there exists an open neighbourhood $V$ of $x$ such that for all $x' \in V \cap h^{-1}(\{0\})$,
\[
    f(x) \leq f(x').
\]
We denote by $\mathcal{M}_-(f,h)$ the set of local minima of $f$ subject to $h$. Likewise $x$ is a strict local minima if $h(x) = 0$ and there exists an open neighbourhood $V$ of $x$ such that for all $x' \in V \cap h^{-1}(\{0\})$,
\[
    f(x) \geq f(x') \implies x = x'.
\]
We denote by $\mathcal{M}_{--}(f,h)$ the set of strict local minima. Likewise, $\mathcal{M}_+(f,h)$ is the set of local maxima and $\mathcal{M}_{++}(f,h)$ the set of strict local maxima.
\end{definition}

\begin{definition}
Let $f: U \subset \mathbb{R}^n \to \mathbb{R}$ and $h: U \subset \mathbb{R}^n \to \mathbb{R}^m$ be continuous, with $U$ open. We say that $x \in U$ is a critical point of $f$ (subject to $h$) if $h(x) = 0$ and 
\[
    \nabla_x f(x) \in (\ker J)^\perp,
\]
namely if $x$ satisfies the necessary Lagrange condition. We note $\mathcal{C}(f,h)$ the set of critical points.
\end{definition}

Even if there is no constraint, we note that $f$ could have no local minima in general,
\[
    x \mapsto e^x,
\]
or that the set of local minima may not be bounded,
\[
    x \mapsto \sin x,
\]
or even that the set of local minima may not be closed,
\[
    x \mapsto x^3 \sin \frac{1}{x}.
\]
However, we can consider these cases rather to be rare in applications. Notably, if $f$ is coercive, then the set of local minima is nonempty and since for $x$ solution of \eqref{eq:odescaledflat}, $s \mapsto f(x(s))$ is non-increasing, $x$ is bounded so that we can restrict its stability analysis on a compact. Further, the last problem is really pathological, although the provided example is twice continuously differentiable. Note though that the set of critical points is closed as we can rewrite it,
\[
    \mathcal{C}(f,h) = h^{-1}(\{0\}) \cap (P \nabla_x f)^{-1}(\{0\}),
\]
where $P$ is the orthogonal projection on $\ker J$, and $^{-1}$ denotes here the preimage.

\begin{proposition}
Let $f: U \subset \mathbb{R}^n \to \mathbb{R}$ and $h: U \subset \mathbb{R}^n \to \mathbb{R}^m$ be twice differentiable, with $U$ open. If $f$ is coercive, then for all initial condition there is a unique global solution to \eqref{eq:odescaledflat} (defined for positive times), converging to $\mathcal{C}(f,h)$, furthermore .
%If further the local minima set $\mathcal{M}(f,h)$ is closed,  then for almost any initial condition, the solution converges to a local minimum.
\end{proposition}

\begin{proof}
The proof relies mainly on two lemmas, presented in appendix \ref{app:ode}. Firstly, the escape dilemma (Lemma \ref{lem:escape}) states that the maximal solution to \eqref{eq:odescaledflat} (with a given initial condition) is either defined for all positive times, or leaves any compact.
\medskip

By lemma \ref{lem:proj}, the matrix,
\[
    P \triangleq (I_n - J^\top (J J^\top)^{-1} J),
\]
is the orthogonal projection on $\ker J$ and thus $I_n \succeq P \succeq 0$. Let now $x_0 \in U$. As the function,
\[
    g: x \in U \mapsto - P \nabla_x f(x),
\]
is continuously differentiable, by Picard-Lindel\"of theorem, there exists a unique maximal solution $(I,x)$ satisfying the initial condition $x(0) = x_0$, where $I$ is an open interval containing $0$ and we define $I_+ = I \cap \mathbb{R}_+$.

Then,
\[
    \frac{\mathrm{d}}{\mathrm{d} s} f(x(s)) = - \nabla_x f(x(s))^\top P \nabla_x f(x(s)) \leq 0,
\]
with equality if and only if $\nabla_x f(x(s))$ belongs to $(\ker J)^\perp$, that is if and only if $x(s)$ is a critical point. For now, this means that $x(I_+) \subset f^{-1}((-\infty, f(x_0)]) \triangleq K$ compact, in other words $x$ never leaves $K$, thus is defined for all positive times by the escape dilemma (lemma \ref{lem:escape}).
\bigskip

Secondly, we invoke the Lyapunov wanderers lemma (lemma \ref{lem:wanderers}). We note that,
\[
    g^{-1}(\{0\}) = \mathcal{C}(f,h) \subset M \subset (\nabla_x f^\top g)^{-1}(\{0\}) = \mathcal{C}(f,h).
\]
Therefore, by lemma \ref{lem:wanderers},
\[
    d(x(s), \mathcal{C}(f,h)) \to_s 0. \qedhere
\]
\end{proof}

Additionally, we should like to add that strict local maxima are unstable, this can be simply seen by reversing time and using a classic Lyapunov argument, while strict local minima are stable.

\begin{proposition}
If $x \in \mathcal{M}_{++}(f,h)$, $x$ is `repellent', while if $x \in \mathcal{M}_{--}(f,h)$, $x$ is stable, and asymptotically stable if further it is isolated from $\mathcal{C}(f,h)$.
\end{proposition}

\begin{proof}
This simply relies on Lyapunov's theorem, with $f$ as a Lyapunov function as previously seen. Let $x^* \in \mathcal{M}_{--}(f,h)$, then let $V$ be an open neighbourhood of $x^*$ such that for all $x \in V \cap h^{-1}(\{0\})$,
\[
    f(x^*) \geq f(x) \implies x = x^*.
\]
Without loss of generality $V$ is bounded. Let then,
\[
    \epsilon \triangleq \max_{x \in \bar{V}} f(x) - f(x^*),
\]
and define the level set,
\[
    V' \triangleq f^{-1}([f(x^*), f(x^*) + \nicefrac{\epsilon}{2})) \cap \bar{V}.
\]
If $x_0 \in V'$, then $\varphi^t(x_0) \in V'$ for all positive times $t$, thus $x^*$ is stable. If further $x^*$ is isolated, without loss of generality $V$ does not contain any point of $\mathcal{C}(f,h)$ other than $x^*$ itself, then by Lyapunov wanderers lemma applied to $\bar{V}'$, the non-wandering set is $M = \{x^*\}$, and thus for all $x_0 \in \bar{V}'$,
\[
    \varphi^t(x_0) \to_t x^*.
\]
Finally, if $x^*$ is a strict maximum, it is repellent in the sense that for all open neighbourhood $V$ of $x^*$, there exists a neighbourhood $V'$ of $x^*$ such that $x_0 \not \in V$ implies that for all $t \geq 0$, $\varphi^t(x_0) \not \in V'$. Again, this is seen easily as $f$ is a Lyapunov function, with an application of the intermediate value theorem.
\end{proof}

\clearpage

\newpage
\section{Properties of the Modified Continuous Newton Method}
\label{appendix-RammTheorem}

Let us consider the problem of solving such a system in its generality, wherein we have Banach spaces $X, Y$ and a map $F: X \to Y$:

\begin{align}
\label{Rammoriginal}
F(u) = f
\end{align}

Recall that map $F$ is called Fr{\' e}chet differentiable at $u \in X$ if there exists a bounded linear operator $A: X \to Y$ such that

\begin{align}
\lim _{\|h\|\to 0}{\frac {\|F(u+h)-F(u)-Ah\|_{Y}}{\|h\|_{X}}}=0.
\end{align}

and let us use $F'$ to denote the Fr{\' e}chet derivative.
Instead of using continuous-time Newton method:

\begin{align}
\label{RammNewton}
\dot u &= -[F'(u)]^{-1} [F(u) - f]  \\
\dot u &= u_0
\end{align}

for finding the zeros of $F(u)$, which involves the inversion, we may consider a higher-dimensional problem in $(u, Q)$:

\begin{align}
\label{Rammapprox2}
    \dot u & = - Q[F(u) - f] \\
    u(0) & = u_0 \\
    \dot Q & = -[(F'(u))^* F'(u)] Q + (F'(u))^* \\
    Q(0) & = Q_0
\end{align}

This approach may seem heuristic at first, but notice that:

\begin{assumption}[Bounded Frechet Derivatives]
\label{Ramm93} 
Consider the $j$th Frechet derivatives $F^{(j)}(u)$ at $u$ bounded from above for all $u$ within a ball $B(u_0, R)$ of radius $R$ centered at $u_0$:
\begin{align}
    \sup_{u \in B(u_0, R) } \| F^{(j)}(u) \| \le M_j(R), \quad 0 \le j \le 2
\end{align}
\end{assumption}

\begin{assumption}[Well-Conditioned Inverse]
\label{Ramm92}
\begin{align}
\sup_{u \in B(u_0, R) } \| F'(u) \|^{-1} \le m(R)
\end{align}
\end{assumption}

\begin{theorem}[Ramm's Theorem for Well-Conditioned Problems]
Under Assumptions \ref{Ramm92}--\ref{Ramm93},
consider an equation \eqref{Rammoriginal} with solution $y$.
Let us have $u_0$ sufficiently close to $y$,
and $Q_0$ sufficiently close to $(F'(y))^{-1}$.
Then problem \eqref{Rammapprox2} has a unique global solution
and there exists a time $t$ and $u(t)$ such that
$u(t) = y$, i.e., 
\begin{align}
    \lim_{t \to \infty} \| u(t) - y \| = 0 \\
    \lim_{t \to \infty} \| Q(t) - (F'(y))^{-1} \| = 0
\end{align}
\end{theorem}

One can relax Assumptions \ref{Ramm92} at the expense of a more technical result. One can also prove exponential rate of convergence \cite[Theorem 3.1]{Airapetyan1999}.

%\begin{theorem}[Ramm's Theorem for Ill-Conditioned Problems]
%Under Assumption \ref{Ramm93},consider an equation \eqref{Rammoriginal} with solution $y$.Let us have $u_0$ sufficiently close to $y$,and $Q_0$ sufficiently close to $(F'(y))^{-1}$.Then problem \eqref{Rammapprox} has a unique global solutionand there exists a time $t$ and $u(t)$ such that $u(t) = y$, i.e., 
%\begin{align}
%    \lim_{t \to \infty} \| u(t) - y \| = 0 \\
%    \lim_{t \to \infty} \| Q(t) - (F'(y))^{-1} \| = 0
%\end{align}
%\end{theorem}

\end{document}